\documentclass[a4paper,11pt,reqno]{amsart}
\usepackage[a4paper, left=1in, right=1in, top=0.75in, bottom=0.7in]{geometry}
\usepackage{amsfonts}
\usepackage{dsfont}
\usepackage{amsmath}
\usepackage{array}
\usepackage{amssymb}
\usepackage{mathrsfs}
\usepackage{array}
\usepackage{comment}
\usepackage{csquotes}
\usepackage{enumitem}
\usepackage{mdframed}
\usepackage{xurl}
\usepackage{todonotes}

\newmdtheoremenv{lem}{Lemma}
\newmdtheoremenv{cor}{Corollary}
\newmdtheoremenv{rem}{Remark}
\usepackage{color}
\usepackage{multirow}
\allowdisplaybreaks

\usepackage[dvipsnames]{xcolor}
\usepackage[colorlinks = true,
            linkcolor = blue,
            urlcolor  = gray,
            citecolor = blue,
            anchorcolor = blue]{hyperref}
\renewcommand\eqref[1]{(\ref{#1})} 

\numberwithin{equation}{section}
\theoremstyle{plain}
\newtheorem{theorem}{Theorem}[section]
\newtheorem{proposition}[theorem]{Proposition}
\newtheorem{corollary}[theorem]{Corollary}
\newtheorem{lemma}[theorem]{Lemma}
\theoremstyle{definition}

\newtheorem{remark}[theorem]{Remark}

\DeclareMathOperator{\supp}{supp}

\usepackage{xurl} 
\usepackage{ragged2e}

\begin{document}

   \title[Traveling waves for nonlocal DNLS]
   {Traveling waves for nonlocal derivative nonlinear Schr\"odinger equations: a variational characterization}

\author[A. Esfahani]{Amin Esfahani}
\address{
  Amin Esfahani:
  \endgraf
  School of Computing and Artificial Intelligence (SCAI)
  \endgraf
Nazarbayev University, Astana, Kazakhstan
  \endgraf
  {\it E-mail address:} {\rm  amin.esfahani@nu.edu.kz}
  }

\author[A. Kairzhan]{Adilbek Kairzhan}
\address{
  Adilbek Kairzhan:
  \endgraf
  School of Computing and Artificial Intelligence (SCAI)
  \endgraf
Nazarbayev University, Astana, Kazakhstan
  \endgraf
  {\it E-mail address:} {\rm  akairzhan@nu.edu.kz}
  }

\author[M. Karazym]{Mukhtar Karazym}
\address{
  Mukhtar Karazym:
  \endgraf
  Private Institution Nazarbayev University Research Administration, Astana, Kazakhstan
  \endgraf
\&
  \endgraf
School of  Artificial Intelligence and Data Science
\endgraf
Astana IT University, Astana, Kazakhstan
\endgraf
  {\it E-mail address:} {\rm  mukhtar.karazym@nu.edu.kz}
  }

\subjclass[2020]{}

\begin{abstract} 
We establish several existence results for traveling-wave solutions of the nonlocal derivative nonlinear Schr\"odinger equation with general coefficients by
variational methods. We study associated minimization problems in the subcritical and critical cases and prove the existence of a minimizer in each case. Finally, we derive Pohozaev-type identities and use them to establish corresponding nonexistence results.
\end{abstract}
\maketitle
\setcounter{tocdepth}{3}
\setcounter{secnumdepth}{3}

\section{Introduction}

\subsection{The model} 
We consider the nonlocal derivative nonlinear Schr\"odinger (nonlocal DNLS) equation given by
\begin{equation}
\label{nonlocal DNLS}
i u_t - u_{xx} - b |u|^2 u 
+ i \alpha |u|^2 u_x 
+ i \beta u^2 \bar{u}_x 
+ \gamma u  \partial_x \mathcal{H}(|u|^2) = 0,
\end{equation}
where $ b, \alpha, \beta, \gamma \in \mathbb{R}$ are parameters, and $ \mathcal{H} $ denotes the Hilbert transform defined by
\begin{equation*}
\mathcal{H}u(x,t) = \frac{1}{\pi} \,  \mathrm{p.v.}  \int_{\mathbb{R}} 
\frac{u(y,t)}{x - y}   dy,
\end{equation*}
see also \eqref{hilbert-fourier-defn} for its definition in the Fourier convention. 

The equation \eqref{nonlocal DNLS} arises while investigating water waves in the modulational regime when the surface waves of the water are approximated by slow-modulations of near-monochromatic waves. This was shown by Lo and Mei in \cite{LM1985}, where they derived \eqref{nonlocal DNLS} from a more general model, known as the Dysthe equation, after an appropriate change of variables. Later, the equation was studied in more detail by Fedele and Dutykh in \cite{fedele2012}. It was shown that for generic values of the parameters, the solution to \eqref{nonlocal DNLS} conserves the mass and momentum functional 
$$\mathcal{M}(u)=\frac{1}{2}\int_{\mathbb{R}}|u(x,t)|^2dx, \qquad 
\mathcal{P} (u) = \int _{\mathbb{R}}\left( \operatorname{Im} (u_x \overline u) - \frac{1}{2} \beta |u|^4 \right)dx
$$
as well as the functional 
$$
\mathcal{E} (u) = \int_{\mathbb{R}}\left( |u_x|^2 -\frac{b}{2} |u|^4 + \frac{\beta (\alpha+\beta)}{6} |u|^6 - \frac{\alpha+\beta}{2} |u|^2 \operatorname{Im} (u_x \overline u) + \frac{\gamma}{2} |u|^2 \partial_x\mathcal{H}  (|u|^2) \right)dx.
$$
It is worth noting that the equation \eqref{nonlocal DNLS} is not Hamiltonian in general since the original Dysthe equation, obtained in \cite{D79} and used as a starting point in \cite{LM1985}, is not Hamiltonian. However, in a special scenario when $\beta = 0$, there is a Hamiltonian formulation for \eqref{nonlocal DNLS}, which can be derived from a Hamiltonian version of Dysthe equation obtained previously in \cite{CGS21, GT11, GKS22}. Indeed, when the parameter $\beta$ vanishes, the equation \eqref{nonlocal DNLS} has a following canonical form
\begin{equation*}
\begin{aligned}
\partial_t \begin{pmatrix}
u \\ \overline u
\end{pmatrix} = \begin{pmatrix}
0 & -i \\ i & 0
\end{pmatrix} \begin{pmatrix}
\delta_u E \\ \delta_{\overline u} E
\end{pmatrix}
\end{aligned}
\end{equation*}
with a Hamiltonian given by $E := - \mathcal{E}(u; \beta = 0)$. 

For a particular choice of parameters, the equation \eqref{nonlocal DNLS} can be reduced to the cubic NLS and the derivative NLS equations, both of which have been a subject of extensive research in the past several decades. 
PDEs of recent interest in the literature are recovered when the coefficient $\gamma$ of the nonlocal term is nonzero. For example, in case $\alpha = \beta = -\gamma$, the equation \eqref{nonlocal DNLS} is equivalent to 
\begin{equation}
\label{de Moura}
iu_t - u_{xx} = b |u|^2 u 
- i\beta u (1 + i\mathcal{H}) (|u|^2)_x = 0.
\end{equation}
This is the limiting case of the intermediate NLS equation 
\begin{equation}
\label{inls}
iu_t - u_{xx} = b|u|^2 u- i\beta u (1 + i\mathcal{T}_h) (|u|^2)_x = 0,   
\end{equation}
where $\mathcal{T}_h$ acts as
$$
\mathcal{T}_h f(x):=\frac{1}{2 h} \text { p.v. } \int_{\mathbb{R}} \operatorname{coth}\left(\frac{\pi(x-y)}{2 h}\right) f_n(y) d y, \qquad 0<h<\infty. 
$$
The equation \eqref{inls} was derived by Pelinovsky and Grimshaw in a series of papers \cite{pelinovsky1995, pelinovsky1995g, Pelinovsky1996} as a model to describe the evolution of nearly monochromatic (or quasi-harmonic) internal waves in a stratified fluid of finite depth. The limit $h \to \infty$ formally leads to a deep water case, and the equation \eqref{inls} becomes \eqref{de Moura}. 

The well-posedness of the Cauchy problem for \eqref{de Moura}--\eqref{inls} has been studied by de Moura in \cite{deMoura2007}, where the local well-posedness was established in $H^s(\mathbb{R};\mathbb{C})$ with $s\geq 1$ for sufficiently small initial data. Later, the result was extended to $H^s(\mathbb{R};\mathbb{C})$ with $s>\frac{1}{2}$ by de Moura and Pilod in \cite{deMoura2010} for arbitrarily large initial data. Very recently, this was further improved by Chapouto, Forlano and Laurens in \cite{chapouto2025}, where the local well-posedness was achieved for $H^s(\mathbb{R};\mathbb{C})$ with $s>\frac{1}{4}$. In \cite{chapouto2025}, the authors also proved a global well-posedness in $H^s(\mathbb{R};\mathbb{C})$ for any $s>\frac{1}{4}$ under additional smallness assumptions on the $L^2$-norm size of the solution. 
We refer to the introduction of \cite{chapouto2025} for a more thorough discussion about the development of the well-posedness theory for \eqref{de Moura}, \eqref{inls}, and related models. In contrast, the existence of traveling wave solutions has received less attention and is the main focus of the present paper.

\subsection{Main result and motivation}
The goal of the paper is to investigate the existence of traveling wave solutions to \eqref{nonlocal DNLS}. 
Such solutions are known to exist for many dispersive PDEs and propagate at constant speed while preserving their profile. They often capture fundamental mechanisms underlying the dynamics of the model, and play an important role in understanding the long-time dynamics of solutions. 
In many dispersive systems, it is observed that arbitrary finite-energy solutions asymptotically (in time)  decompose into a superposition of localized traveling (solitary) waves and a dispersive radiation term, a scenario commonly referred to as \textit{the soliton resolution conjecture}. Within this framework, solitary waves represent the nonlinear core of the long-time dynamics, while the remaining component disperses over time.  

We are interested in the equation \eqref{nonlocal DNLS} in the presence of a nonlocal nonlinear term. In contrast to the classical local NLS, nonlocal nonlinearities introduce additional analytical challenges, including the lack of pointwise structure and a more delicate variational analysis.

In this paper, we rigorously prove the existence of traveling wave solutions to \eqref{nonlocal DNLS} for a wide range of parameters $(b, \alpha, \beta, \gamma)$. Previously, the existence of such solutions was numerically predicted by Fedele and Dutykh in \cite{fedele2012} for a particular choice of parameters $(b, \alpha, \beta, \gamma)$. Another nonlocal PDE model known as the Calogero-Moser DNLS, which is relevant to our results, was studied by G\'erard and Lenzmann in \cite{gerard2024}. More details on this and other related PDEs are provided in Subsection \ref{subsection-previous-results}. We point out that our results recover certain known cases and extend the existing theory to additional parameter regimes.

We briefly outline the form of the traveling waves we are interested in, and we refer to Section \ref{section 2} for more details. We first apply the gauge
transformation
\begin{equation}
	\label{gauge-transform-intro}
	\phi(x,t)
	=
	u(x,t)\exp\left(
	-\frac{i(\alpha+\beta)}{4}
	\int_{x_0}^{x}|u(y,t)|^2 dy
	\right),
\end{equation}
where
$$
x_0=
\begin{cases}
	0, & \text{in the critical case},\\
	-\infty, & \text{in the subcritical case}.
\end{cases}
$$
In the subcritical case, the integral from $-\infty$ is well defined, whereas in the critical case we use the
finite lower limit $x_0=0$ because $L^2$-integrability is not available.

We then use the traveling-wave ansatz
\begin{equation*}
\phi(x,t) = \exp \left(i\omega t - i\frac{c}{2}(x-ct) \right) \psi (x-ct), 
\end{equation*}
and seek profiles $\psi$ satisfying
\begin{equation}
\label{ode-for-psi-intro}
-\psi^{\prime \prime}
- \left(\omega + \frac{c^2}{4}\right) \psi
+  A |\psi|^2 \psi
+ B |\psi|^4 \psi
+ G \psi   \operatorname{Im}(\psi^{\prime} \bar{\psi})
+ \gamma \psi   (\mathcal{H}(|\psi|^2))^{\prime}
= 0,
\end{equation}
where
\begin{equation}\label{A, B, G}
 A := -b + \frac{c(\alpha - \beta)}{2}, \quad B := \frac{(\alpha + \beta)(-3\alpha + 5\beta)}{16},\quad G := \beta - \alpha.
\end{equation}
We show the existence and nonexistence results for solutions to \eqref{ode-for-psi-intro} for a range of parameters $\left(\omega + c^2/4,  A, B, \gamma\right)$. Table \ref{table} summarizes these results.

We note that, after eliminating the induced mean-flow potential from
equation (2.6) of Lo and Mei \cite{LM1985} in the infinite-depth setting,
the resulting equation corresponds in our notation to
$$
b=1,\qquad
\alpha=8\varepsilon,\qquad
\beta=0,\qquad
\gamma=2\varepsilon.
$$
In this case, $B<0$ and $\gamma>0$, and therefore
Theorem \ref{subcritical: B<0 and gamma>=0} applies. The critical case
$$
\omega+\frac{c^2}{4}=0
$$
is also covered by Theorem \ref{main thm critical} whenever $A>0$. This
corresponds to 
$c>1/(4\varepsilon)$ for \cite{LM1985}.

We also note that, since the Lagrange multiplier in
Theorem \ref{thm: fixed quartic term} is negative, it yields a particular case of
Theorem \ref{main thm critical}. We nevertheless include this formulation 
separately in Section \ref{section 6}, because it provides an alternative variational characterization of 
the corresponding traveling waves.

As a result, once the existence of an $H^1$ or
$\dot H^1\cap L^4$ solution of \eqref{ode-for-psi-intro} is established,
the corresponding traveling-wave solution of \eqref{nonlocal DNLS} is
given, in the subcritical case, by
\begin{equation}
\label{travel-wave-form}
\begin{aligned}
u(x,t)
=
\psi(x-ct)\exp\bigg(
i\omega t-\frac{ic}{2}(x-ct)+\frac{i(\alpha+\beta)}{4}
\int_{-\infty}^{x-ct}|\psi(y)|^2dy
\bigg),
\end{aligned}
\end{equation}
whereas in the critical case it is given by
\begin{equation}
\begin{aligned}
u(x,t)
=
\psi(x-ct)\exp\bigg(
i\omega t-\frac{ic}{2}(x-ct)
+\frac{i(\alpha+\beta)}{4}
\int_{0}^{x-ct}|\psi(y)|^2dy
\bigg).
\end{aligned}
\end{equation}
Here $\psi$ is referred to as the ground state and arises from the
variational minimization problems considered in this paper. 

\bigskip
\begin{table}[ht]
\centering
{
\renewcommand{\arraystretch}{1.5}

\begin{tabular}{|m{2cm}|>{\raggedright\arraybackslash}m{2.8cm}|m{3.2cm}|m{2.6cm}|m{2.3cm}|}
\hline
$-\left(\omega+c^2/4\right)$ 
& $A$ 
& $B$ 
& $\gamma$ 
& main result \\ 
\hline

\multirow{3}{=}{positive (subcritical case)}
& any
& negative
& nonnegative
& \multirow{2}{=}{Theorem \ref{subcritical: B<0 and gamma>=0}, existence}
\\
\cline{2-4}

& negative
& zero
& nonnegative
&
\\
\cline{2-5}

& nonpositive, derived from the Lagrange multiplier
& zero
& nonpositive, derived from the Lagrange multiplier $(A,\gamma)\neq(0,0)$ 
& Theorem \ref{thm: existence}, existence
\\
\hline

\multirow{2}{=}{zero (critical case)}
& positive
& \multirow{2}{*}{negative}
& \multirow{2}{*}{nonnegative}
& Theorem \ref{main thm critical}, existence
\\
\cline{2-2}\cline{5-5}

& positive, derived from the Lagrange multiplier; depends on $B$, $\gamma$ and $q$; 
&
&
& Theorem \ref{thm: fixed quartic term}, existence
\\
\hline

any
& nonnegative
& nonnegative
& positive
& \multirow{2}{=}{Theorem \ref{thm: nonexist.}, nonexistence}
\\
\cline{1-4}

nonpositive
& nonpositive
& nonpositive
& any
&
\\
\hline

zero (critical case)
& zero
& $B>\gamma^2/4$
& negative
& Theorem \ref{thm:critical-nonexistence-negative-gamma}, nonexistence
\\
\hline

\end{tabular}
}

\vspace{0.5cm}

\caption{
Throughout the table, $q$ denotes a constraint level given in
Sections \ref{section 5} and \ref{section 6}.
The coefficients $A$ and $B$ are given in \eqref{A, B, G}.
}
\label{table}
\end{table}

\subsection{Related results}
\label{subsection-previous-results}

Here we recall several relevant existence results for traveling wave solutions that are available in the literature, although some of the corresponding cases fall outside the scope of the present work.

\textbf{\textit{Consider the case $b=0$, $\gamma=-1$ and $\alpha=\beta=1$ in the equation \eqref{nonlocal DNLS}.}} Then, defining $v(x, t):=u(x, -t)$, we can rewrite \eqref{nonlocal DNLS} as 
\begin{equation}
\label{NDNLS-}
i v_t+v_{xx} - i v (1+i\mathcal{H})\left(|v|^2\right)_x=0,
\end{equation}
which is known as the Calogero–Moser DNLS equation. This equation has a Lax pair structure and is completely integrable, and therefore, it is feasible to expect that its solitary wave solutions can be explicitly found. Indeed, the equation was studied recently by G\'erard and Lenzmann in \cite{gerard2024}, and it was shown that \eqref{NDNLS-} admits a family of traveling wave solutions given by 
\begin{equation}
\label{travel-waves-gl24}
v(x,t)=e^{i \theta+i \eta x-i \eta^2 t} \lambda^{1 / 2} \mathcal{R}(\lambda(x-2 \eta t)+y), 
\end{equation}
where 
$$\mathcal{R}(x)=\frac{\sqrt{2}}{x+i}$$
with some $\theta \in[0,2 \pi)$, $y \in \mathbb{R}$, $\lambda>0$ and $\eta \in \mathbb{R}$. Although this parameter choice is not covered by the existence theorems
of the present paper, the explicit solutions in
\eqref{travel-waves-gl24} are compatible with the traveling-wave
representation \eqref{travel-wave-form}.  Indeed, taking $\lambda=1$ and
$y=0$ in \eqref{travel-waves-gl24}, and recalling that
$v(x,t)=u(x,-t)$, we obtain
\begin{equation*}
\label{CM-nonlocalNLS-waves-rel}
v(x, t) = \psi(x + c t)  \exp \left( 
-i \omega t - i \frac{c}{2}(x + c t) + \frac{i}{2} 
\int_{-\infty}^{x + c t} |\psi(y)|^2   dy 
\right)
\end{equation*}
with $\theta=\pi\pmod{2\pi}$, $\omega =-\eta^2$, $c =-2\eta$ and
\begin{equation}
\label{CM DNLS sol.}
\psi(x) =\frac{\sqrt{2}}{\sqrt{x^2+1}} \in  H^1(\mathbb{R};\mathbb{C}).
\end{equation}
Note that \eqref{CM DNLS sol.} satisfies
\begin{equation}\label{eq:GM-stationary}
-\psi''
+\frac{1}{4}|\psi|^4\psi
-\psi\big(\mathcal{H}(|\psi|^2)\big)'
=0.
\end{equation}
More generally, if
$$
\gamma<0,
\qquad
B=\frac{\gamma^2}{4},
$$
then the change of amplitude
$$
\psi=\sqrt{-\gamma}\,\widetilde\psi
$$
reduces
\begin{equation}\label{eq:GM-stationary-rescaled}
-\widetilde\psi''
+B|\widetilde\psi|^4\widetilde\psi
+\gamma\widetilde\psi
\big(\mathcal{H}(|\widetilde\psi|^2)\big)'
=0
\end{equation}
to \eqref{eq:GM-stationary}. Hence, the solution of \eqref{eq:GM-stationary-rescaled} is
$$
\widetilde\psi(x)
=
\sqrt{\frac{2a}{-\gamma}}\,
\frac{1}{\sqrt{a^2+(x-x_0)^2}},
\qquad
a>0,\quad x_0\in\mathbb{R}.
$$
We complement this result by proving the nonexistence of nontrivial solutions in the regime
$$
B>\frac{\gamma^2}{4},\quad \gamma<0,
$$
see Theorem \ref{thm:critical-nonexistence-negative-gamma}.

\textbf{\textit{Consider the special case of our results, where $\gamma = b = 0$, and either $(\alpha,\beta)=(-1,0)$ or $(\alpha,\beta)=(-2,-1)$ in the equation \eqref{nonlocal DNLS}.}} Then, after the gauge transformation \eqref{gauge-transform-intro} with $x_0=-\infty$ and defining $v(x, t):=\phi(x,-t)$, the equation \eqref{nonlocal DNLS} becomes
\begin{equation*}\label{DNLS}
i v_t+ v_{xx}+\frac{1}{2} i|v|^2 v_x-\frac{1}{2} i v^2 \bar{v}_x+\frac{3}{16}|v|^4 v=0.
\end{equation*}
This equation admits a two-parameter family of solitary wave solutions, which are unique in $H^1(\mathbb{R};\mathbb{C})$ up to the phase rotation and spatial translation  \cite{Miao2018}, and are given by 
$$v(x,t)=\exp{\left(i\omega t + i \frac{c}{2}(x-ct) \right)} \phi_{\omega, c}(x-ct),$$ 
where
\begin{equation*}
\label{DNLS sol.}
\renewcommand{\arraystretch}{2}
\phi_{\omega, c}(x):=
\left\{
\begin{array}{ll}
\displaystyle
\left(\frac{\sqrt{\omega}}{4 \omega-c^2}
\left(\cosh\left(\sqrt{4 \omega-c^2}  x\right)
-\frac{c}{2 \sqrt{\omega}}\right)\right)^{-1 / 2},
& 4 \omega>c^2, \\
\displaystyle
2 \sqrt{c} (c^2 x^2+1)^{-1 / 2},
& 4 \omega=c^2,\; c>0.
\end{array}
\right.
\end{equation*}

Some of our results build on ideas from previous works, for instance, \cite{Colin2006} and \cite{Miao2018}, but the
present study uses a different gauge and a distinct variational framework,
arising naturally from the form of our equation.  In
particular, when $\gamma=0$, Theorems \ref{subcritical: B<0 and gamma>=0}
and \ref{main thm critical} yield traveling-wave solutions to
$$
iu_t-u_{xx}-b|u|^2u
+i\alpha|u|^2u_x
+i\beta u^2\overline{u}_x=0
$$
of the form
$$
u(x,t)=e^{i\omega t}\phi(x-ct)
$$
in the parameter regimes covered by these theorems. Explicit traveling-wave
solutions of this local equation for general values of $\alpha$ and $\beta$
were previously obtained in \cite{Kondo1997}.

There exist many additional results in the literature establishing the existence of special classes of solutions to PDEs related to our models \eqref{nonlocal DNLS} or \eqref{ode-for-psi-intro}. However, the type of solutions constructed in those works differs substantially from the traveling wave solutions considered in this present paper. For this reason, although these results are closely related in spirit, we only briefly mention a few of such works and refer the reader to the references therein for a more comprehensive overview. 
In particular, we point out a recent work by Chen and Pelinovsky in \cite{Chen2025}, where the existence of traveling periodic waves and breathers was studied for the nonlocal derivative NLS equation. Among results for the \textit{local} models, Lu et al. in \cite{lu2016} obtained bright and dark soliton solutions to the equation in case $\alpha = 2\beta$ and $\gamma = 0$ in \eqref{nonlocal DNLS}. The equations similar to \eqref{ode-for-psi-intro} were studied by in \cite{ming2008, yue2016} where the existence of algebraic, bell-type, kink-type and sinusoidal traveling wave solutions was established.

\subsection{Structure of the paper}

In Section \ref{section 2}, we apply a gauge transformation to simplify the derivative cubic terms in \eqref{nonlocal DNLS}, so this nonlinearity can be reduced to a single derivative term with a simpler structure. We then impose a traveling-wave ansatz and factor out the spatial phase $e^{-i c x / 2}$, yielding an ODE with a nonlocal term. Following \cite{Colin2006}, we further simplify this ODE so that it no longer depends on the complex conjugate.

In Sections \ref{section 3} and \ref{section 4}, we consider subcritical and critical minimization problems involving the Nehari functional. In both cases, the action functional is unbounded from below; therefore, we replace it with a positive functional and enlarge the constraint set. We then prove the equivalence of the resulting minimization problems. A similar approach is used in the context of a two-parameter family of solitary wave solutions for the derivative NLS; see, for example,  \cite{Fukaya2017} and \cite{Miao2018}.

Since the equation \eqref{nonlocal DNLS} contains many parameters, several different cases arise. In Section \ref{section 5}, we study a minimization problem for a coercive functional under constraints involving the nonlocal and local quartic term. The main tool is Lions’ concentration-compactness principle, which yields the existence of a minimizer. Related applications of this method can be found in \cite{Borluk2026, Chen2011,  Duran2024, Kabakouala2018, Levandosky1999}.

In Section \ref{section 6}, the action functional takes negative values, and the quartic term appears as a constraint. The existence of a minimizer is also derived from Lions’ concentration-compactness principle. Similar phenomena have been
observed in other contexts, see e.g. \cite{Albert1999, Pava1999, kichenassamy, Nguyen2011}.

In Section \ref{section 7}, we establish the nonexistence of smooth traveling wave solutions with weighted decay conditions by using Pohozaev-type identities.

We point out that the present paper focuses on the existence and nonexistence results of traveling
wave solutions. 
For the ground states obtained in our existence
theorems with $\gamma\neq0$, uniqueness (up to the phase rotation and spatial translation symmetries) remains open; see e.g.
\cite{laire2022}.  In fact, proving uniqueness for nonlocal equations can be quite challenging; see, e.g., \cite{Amick1991} and \cite{Frank2013}.

The stability or instability of the constructed traveling waves is not
addressed in the present work and remains to be studied. In the
Calogero--Moser DNLS case \eqref{NDNLS-}, which is not covered by our existence
theorems, Hogan and Kowalski \cite{Hogan2024} proved a restricted stability result modulo
phase, translation, and scaling, and constructed perturbations of the soliton
whose Sobolev norms become unbounded for every $s>0$.

\subsection{Notation and preliminaries}

The expression $a \lesssim d$ means $a \leq C d$ for some positive independent constant $C$.
We denote by $$\langle f, g\rangle:=\int_{\mathbb{R}} f \bar{g}dx$$ the $L^2(\mathbb{R};\mathbb{C})$-inner product of  $f$ and $g$. 

Using the standard notation for the spatial Fourier transform
$$
\widehat{f}(\xi):=\frac{1}{\sqrt{2\pi}}\int_{\mathbb{R}} e^{-2\pi i x  \xi} f(x)dx,
$$
we denote by
$$
|f|_{\dot{H}^s}:=\left(\int_{\mathbb{R}}|\xi|^{2 s}|\hat{f}(\xi)|^2 d \xi\right)^{1 / 2}
$$
the   seminorm of the homogeneous Sobolev space $\dot{H}^s(\mathbb{R};\mathbb{C})$ for some $s\in \mathbb{R}$.
We denote by
$$
\|f\|_{H^s}:=\left(\int_{\mathbb{R}}\langle\xi\rangle^{2 s}|\hat{f}(\xi)|^2 d \xi\right)^{1 / 2}
$$
the   norm of the inhomogeneous Sobolev space $H^s(\mathbb{R};\mathbb{C})$ for some $s\in \mathbb{R}$, where $$\langle\xi\rangle:=\left(1+|\xi|^2\right)^{1 / 2}.$$

Since this paper addresses both subcritical and critical cases, it is
necessary to introduce the function spaces used for the associated
minimization problems. We begin with the subcritical case.
Let $\omega,c\in\mathbb{R}$ satisfy $4\omega<-c^2$. We equip
$H^1(\mathbb{R};\mathbb{C})$ with the equivalent norm
$$
\|\varphi\|_{\widetilde H_c}^2
:=
\|\psi'\|_{L^2}^2
-\left(\omega+\frac{c^2}{4}\right)\|\psi\|_{L^2}^2,
\qquad
\psi(x):=e^{icx/2}\varphi(x),
$$
and denote the resulting Hilbert space by $\widetilde H_c$. We see that the map
$$
\varphi \mapsto e^{icx/2}\varphi
$$
is a Hilbert space isomorphism from $\widetilde H_c$ onto
$H^1(\mathbb{R};\mathbb{C})$.

In the critical case, the associated minimization problem must be considered in a different space. 
To this end, we define a Banach space
$$
X_c := \left\{ \varphi \in \mathcal{S}'(\mathbb{R};\mathbb{C}) : \varphi(x) = \exp\left(-\frac{icx}{2}\right)\psi(x) 
\ \text{with } \psi \in \dot{H}^1(\mathbb{R}; \mathbb{C})\cap L^4(\mathbb{R}; \mathbb{C}) \right\},
$$
equipped with the norm
$$
\|\varphi\|_{X_c} := |\psi|_{\dot{H}^1} + \|\psi\|_{L^4}.
$$
This space will be used to treat the critical cases in Sections \ref{section 4} and \ref{section 6}.

Let us recall some properties of the Hilbert transform.  We can express $\mathcal{H}$ as a Fourier multiplier operator given by
\begin{equation}
\label{hilbert-fourier-defn}
\widehat{\mathcal{H}(f)}(\xi)=-i \operatorname{sgn}(\xi) \hat{f}(\xi).
\end{equation}
Moreover, it is known that
$$
\partial_x \mathcal{H}=\mathcal{H} \partial_x=|D|
$$
and
$$
\int_{\mathbb{R}} u  \partial_x\mathcal{H} u dx=\int_{\mathbb{R}}\big||D|^{1/2} u\big|^2dx=|u|_{\dot{H}^{1/2}}^2.
$$
The operator $|D|$ can be written as
\begin{equation}\label{eq:|D|}
|D| u(x)
=c\, \mathrm{p.v.}\int_{\mathbb{R}}
\frac{u(x)-u(y)}{|x-y|^{2}} dy.
\end{equation}

Lastly, we present the concentration compactness lemma by P.L. LionsSchool of Computing and Artificial Intelligence (SCAI) \cite{Lions1984}.

\begin{lemma}\label{lem:CC}
Let $\{\rho_n\}$ be a sequence of nonnegative functions in $L^1(\mathbb{R};\mathbb{R})$ such that
$$
\int_{\mathbb{R}} \rho_n(x) dx = L > 0
\qquad \text{for all } n\in\mathbb{N}.
$$
Then, up to a subsequence, exactly one of the following alternatives occurs:

\begin{enumerate}
\item[\textnormal{(i)}] \textbf{Compactness.}  
There exists a sequence $\{y_n\}\subset \mathbb{R}$ such that for every $\varepsilon>0$ there exists $R=R(\varepsilon)>0$ with
$$
\int_{|x-y_n|\leq R} \rho_n(x) dx \geq L-\varepsilon
\qquad \text{for all } n.
$$

\item[\textnormal{(ii)}] \textbf{Vanishing.}  
For every $R>0$,
\begin{equation*}
\lim_{n\to\infty}\sup_{y\in\mathbb{R}}
\int_{|x-y|\leq R}\rho_n(x) dx = 0.
\end{equation*}

\item[\textnormal{(iii)}] \textbf{Dichotomy.}  
There exists $\ell\in(0,L)$ such that for every $\varepsilon>0$ there exist $R=R(\varepsilon)>0$, a sequence $R_n\to\infty$, and a sequence $\{y_n\}\subset\mathbb{R}$ such that
$$
\left|\int_{|x-y_n|\leq R}\rho_n(x) dx-\ell\right|<\varepsilon
$$
and
$$
\int_{R<|x-y_n|<R_n}\rho_n(x) dx<\varepsilon
$$
for all sufficiently large $n$.
\end{enumerate}
\end{lemma}

\section{Gauge transformation and traveling wave ansatz}\label{section 2}
Applying the gauge transformation
\begin{equation}\label{gauge1}
\phi(x,t)=\mathcal{G}(u)=u(x,t)\exp\left(-\frac{i(\alpha+\beta)}{4}
\int_{x_0}^x|u(y,t)|^2 dy\right),
\end{equation}
we obtain 
\begin{equation}\label{eq:gauge-equation}
\begin{aligned}
i \phi_t-\phi_{xx}
&+\frac{(\alpha+\beta)(-3\alpha+5\beta)}{16}|\phi|^4\phi + i\frac{\alpha-\beta}{2}|\phi|^2\phi^{\prime}
+ i\frac{\beta-\alpha}{2}\phi^2\bar{\phi}^{\prime}
 \\
&- b|\phi|^2\phi
+ \gamma \phi \partial_x\mathcal{H}\left(|\phi|^2\right)+\frac{\alpha+\beta}{2} \operatorname{Im}\big(\phi_x\left(x_0, t\right) \overline{\phi\left(x_0, t\right)}\big) \phi
=0.
\end{aligned}
\end{equation}
The last term vanishes when $x_0=-\infty$. When $x_0=0$, it can be removed by introducing the additional time-dependent phase
$$
\theta(t):=\frac{\alpha+\beta}{2}\int_0^t  \operatorname{Im}\big(\phi_x\left(x_0, s\right) \overline{\phi\left(x_0, s\right)} ds.
$$
Then
$$
\widetilde{\phi}(x,t)
=
e^{-i\theta(t)} \phi(x,t).
$$
Then $\widetilde{\phi}$ satisfies \eqref{eq:gauge-equation} without the last term. With a slight abuse of notation, we henceforth denote $\widetilde{\phi}$ by $\phi$.

Under the traveling wave ansatz
\begin{equation}\label{wave ansatz}
\phi(x, t) = \exp(i \omega t)   \varphi(x - c t),
\end{equation}
we have
\begin{equation*}
\begin{aligned}
-\omega \varphi - i c \varphi^{\prime} &-\varphi^{\prime\prime}+\frac{(\alpha+\beta)(-3 \alpha+5 \beta)}{16}|\varphi|^4 \varphi\\
& +i\left(\frac{\alpha-\beta}{2}\right)|\varphi|^2 \varphi^{\prime} +i \frac{\beta-\alpha}{2} \varphi^2 \bar{\varphi}^{\prime}-b|\varphi|^2 \varphi+\gamma \varphi   (\mathcal{H}(|\varphi|^2))^{\prime}=0.
\end{aligned}
\end{equation*}
Lastly, setting
\begin{equation}\label{phase rot.}
\varphi(x)=\exp\left(-\frac{icx}{2}\right)\psi(x),
\end{equation}
we arrive at
\begin{equation}\label{ODE with Im()}
-\psi^{\prime \prime}
- \left(\omega + \frac{c^2}{4}\right) \psi
+  A |\psi|^2 \psi
+ B |\psi|^4 \psi
+ G \psi   \operatorname{Im}(\psi^{\prime} \bar{\psi})
+ \gamma \psi   (\mathcal{H}(|\psi|^2))^{\prime}
= 0,
\end{equation}
where \begin{equation*}
 A := -b + \frac{c(\alpha - \beta)}{2}, \quad B := \frac{(\alpha + \beta)(-3\alpha + 5\beta)}{16},\quad G := \beta - \alpha.
\end{equation*}
Combining \eqref{gauge1}--\eqref{phase rot.}, we have
\begin{equation}\label{all transf. subcr.}
\begin{aligned}
u(x, t) = \psi(x - c t)  \exp \bigg( 
i \omega t - i \frac{c}{2}(x - c t)+ \frac{i(\alpha + \beta)}{4} 
\int_{-\infty}^{x - c t} |\psi(y)|^2   dy 
\bigg)
\end{aligned}
\end{equation}
in the subcritical case, and
\begin{equation}\label{all transf. crit.}
\begin{aligned}
u(x, t) = \psi(x - c t)  \exp \bigg( 
i \omega t - i \frac{c}{2}(x - c t)+ \frac{i(\alpha + \beta)}{4} 
\int_{0}^{x - c t} |\psi(y)|^2   dy 
\bigg)
\end{aligned}
\end{equation}
in the critical case.
Following the idea from \cite{Colin2006}, we can further simplify the equation \eqref{ODE with Im()}.
\begin{lemma}
\label{Colin lemma}
Let $\psi$ be a solution to \eqref{ODE with Im()}. Then 
$$
\operatorname{Im}\left(\psi^{\prime} \bar{\psi}\right) = 0.
$$
\end{lemma}

\begin{proof}
Let $\psi$ be a solution to \eqref{ODE with Im()}. Writing $\psi = f + i g$ with $f = \operatorname{Re}(\psi)$ and $g = \operatorname{Im}(\psi)$, we observe that both $f$ and $g$ satisfy
$$
f'' = T(\psi) f\quad \text{and}\quad
g'' = T(\psi) g,
$$
where
$$
T(\psi)
:= - \left(\omega + \frac{c^2}{4}\right)
+  A |\psi|^2 
+ B |\psi|^4 
+ G   \operatorname{Im}\left(\psi^{\prime} \bar{\psi}\right)
+ \gamma   (\mathcal{H}(|\psi|^2))^{\prime}.
$$
Consequently,
$$
(fg'-gf')'=0,
$$
so $fg'-gf'$ is constant. Moreover, since
$f,g\in H^1(\mathbb{R};\mathbb{R})$
(resp. $f,g\in \dot{H}^1(\mathbb{R};\mathbb{R})
\cap L^4(\mathbb{R};\mathbb{R})$), we have
$$
fg'-gf'
\in L^1(\mathbb{R};\mathbb{R})
\quad
\big(\text{resp. }L^{4/3}(\mathbb{R};\mathbb{R})\big).
$$
Therefore, this constant must be zero. Hence,
$$
\operatorname{Im}\left(\overline{\psi}\psi'\right)
=fg'-gf'=0.
$$
\end{proof}
In view of Lemma \ref{Colin lemma}, every solution $\psi$ of
\eqref{ODE with Im()} is, up to multiplication by a constant phase
factor, real-valued. More precisely,
$$
\psi=e^{i\theta}\widetilde\psi,
$$
for some $\theta\in\mathbb R$, where $\widetilde\psi$ is real-valued.
Moreover, $\widetilde\psi$ satisfies
\begin{equation}\label{Second order ODE}
	-\psi^{\prime \prime}
	-\left(\omega+\frac{c^2}{4}\right)\psi
	+A|\psi|^2\psi
	+B|\psi|^4\psi
	+\gamma\psi(\mathcal H(|\psi|^2))'
	=0.
\end{equation}
Consequently,  we may assume without loss of
generality that the stationary profile is real-valued.

\section{Subcritical minimization on the Nehari manifold}\label{section 3}

In this section, we obtain traveling waves of \eqref{nonlocal DNLS} for the subcritical case by minimizing the action functional over the Nehari manifold. The same approach is used in the next section to treat the critical case.

Let $\omega, c \in \mathbb{R}$ satisfy $4 \omega<-c^2$.  Also, assume that
$$
(A\in\mathbb{R},\ B<0,\ \gamma\geq0)
\quad\text{or}\quad
(B=0,\ A<0,\ \gamma\geq0).
$$
Since we have eliminated the $\operatorname{Im}\left(\psi^{\prime} \bar{\psi}\right)\psi$-term in Lemma \ref{Colin lemma}, returning to
$$
\varphi(x) := \exp\left(-\frac{icx}{2}\right) \psi(x)
$$
yields
\begin{equation}\label{nonlocal DNLS trav. wave}
-\varphi^{\prime \prime}
+  A |\varphi|^2 \varphi
+ B |\varphi|^4 \varphi
+ \gamma \varphi \left(\mathcal{H}(|\varphi|^2)\right)^{\prime}
- i c \varphi^{\prime}
- \omega \varphi
= 0.
\end{equation}

We say that $\varphi \in \widetilde{H}_c$ satisfies \eqref{nonlocal DNLS trav. wave} 
if and only if it is a critical point of
\begin{equation}\label{Functional E}
\begin{aligned}
E(\varphi) 
&:= \int_{\mathbb{R}} \bigg(
 \frac{1}{2} \left| \varphi^{\prime}\right|^2
+ \frac{ A}{4} |\varphi|^4
+ \frac{B}{6} |\varphi|^6
+ \frac{\gamma}{4} |\varphi|^2 \left(\mathcal{H}(|\varphi|^2)\right)^{\prime}  \\
&- \frac{\omega}{2} |\varphi|^2
+ \frac{c}{2}   \operatorname{Im}\left(\bar{\varphi}   \varphi^{\prime}\right)
\bigg)  dx.
\end{aligned}
\end{equation}
The mapping $E$ is a $C^2$ functional on $\widetilde{H}_c$. We also define
\begin{equation}\label{functional K}
\begin{aligned}
K(\varphi)  &:=\int_{\mathbb{R}}\bigg(\left| \varphi^{\prime}\right|^2+ A|\varphi|^4+B|\varphi|^6+ \gamma |\varphi|^2\left(\mathcal{H}\left(|\varphi|^2\right)\right)^{\prime}\\
&-\omega|\varphi|^2+c \operatorname{Im}\left(\bar{\varphi} \varphi^{\prime}\right)\bigg)dx.
\end{aligned}
\end{equation}
It is continuously differentiable on $\widetilde{H}_c$. An important observation is that if $\varphi \in \widetilde{H}_c \backslash\{0\}$ is a solution to \eqref{nonlocal DNLS trav. wave}, then it satisfies 
$$
K(\varphi)=0.
$$
Taking these observations together,  we are naturally led to  
\begin{equation}\label{E with K(varphi)=0}
I_{0}:=\inf \left\{E(\varphi):\ \varphi \in \widetilde{H}_c \backslash\{0\} \quad \text{and}\quad   K(\varphi)=0 \right\}.
\end{equation}
For convenience, we split the functional $K$ into the quadratic and nonlinear parts
\begin{equation*}\label{K functional quadratic part}
K^Q(\varphi)  :=\int_{\mathbb{R}}\left(\left| \varphi^{\prime}\right|^2+c \operatorname{Im}\left(\bar{\varphi} \varphi^{\prime}\right)-\omega|\varphi|^2\right)dx
\end{equation*}
and 
\begin{equation*}\label{K functional nonlinear part}
\begin{aligned}
K^N(\varphi):&=K^Q(\varphi)-K(\varphi)\\
&=\int_{\mathbb{R}}\left(- A|\varphi|^4-B|\varphi|^6-\gamma|\varphi|^2\left(\mathcal{H}\left(|\varphi|^2\right)\right)^{\prime}\right)dx.   
\end{aligned}
\end{equation*}
Then 
$$
\begin{aligned}
K^Q(\lambda \varphi) & =\lambda^2 \int_{\mathbb{R}}\left(\left| \varphi^{\prime}\right|^2+c \operatorname{Im}\left(\bar{\varphi} \varphi^{\prime}\right)-\omega|\varphi|^2\right)dx \\
& =\lambda^2\left((1-\alpha)\left| \varphi\right|_{\dot{H}^1}^2+\frac{1}{\alpha}\left\|\alpha \varphi^{\prime}+\frac{c}{2} i \varphi\right\|_{L^2}^2-\left(\omega+\frac{c^2}{4 \alpha}\right)\|\varphi\|_{L^2}^2\right)
\end{aligned}
$$
holds for all $\lambda>0$ and $\alpha \in\left(-c^2/(4 \omega), 1\right)$. Hence, we have proved the following lemma.
\begin{lemma}\label{lem:KQ-zero-limit}
Let $\varphi \in \widetilde{H}_c\setminus \{0\}$.  Then
$$
\lim _{\lambda \rightarrow 0+} K^Q(\lambda \varphi)=0.
$$
\end{lemma}

The next lemma demonstrates the behavior of $K$ about the origin of $\widetilde{H}_c$.

\begin{lemma}\label{lem:K-positive-eventually}
Let  $\left\{\varphi_n\right\}  $ be a bounded sequence in $\widetilde{H}_c \backslash\{0\}$ such that
\begin{equation}\label{K quadratic vanishes}
\lim _{n \rightarrow \infty} K^Q\left(\varphi_n\right)=0.
\end{equation}
Then
$
K\left(\varphi_n\right)>0
$
for sufficiently large $n$.
\end{lemma}

\begin{proof}
Let  $\left\{\varphi_n\right\} $ be a bounded sequence in $\widetilde{H}_c \backslash\{0\}$ with \eqref{K quadratic vanishes}. We can rewrite $K^Q$ in terms of $\psi_n$, that is,
$$
\begin{aligned}
K^Q(\varphi_n) 
&= \int_{\mathbb{R}} \left( \left| \varphi_n^{\prime} \right|^2 + c  \operatorname{Im}\left( \bar{\varphi}_n \varphi_n^{\prime} \right) - \omega |\varphi_n|^2 \right)dx \\
&= \int_{\mathbb{R}} \left( \left|  \left( e^{icx/2}\varphi_n\right)^\prime \right|^2 - \left( \omega + \frac{c^2}{4} \right) |e^{icx/2}\varphi_n|^2 \right)dx\\
&= \int_{\mathbb{R}} \left( \left|  \psi_n^{\prime} \right|^2 - \left( \omega + \frac{c^2}{4} \right) |\psi_n|^2 \right)dx.
\end{aligned}
$$
Then by Gagliardo-Nirenberg inequality  and \eqref{K quadratic vanishes},  
$$
\begin{aligned}
|K^N(\varphi_n)|&=\left|\int_{\mathbb{R}}\left(- A|\varphi_n|^4-B|\varphi_n|^6-\gamma |\varphi_n|^2\left(\mathcal{H}\left(|\varphi|^2\right)\right)^{\prime}\right)dx  \right| 
\\
 &=\left| \int_{\mathbb{R}}\Big[- A|\psi_n|^4-B|\psi_n|^6-\gamma |\psi_n|^2\left(\mathcal{H}\left(|\psi_n|^2\right)\right)^{\prime}\Big]dx\right|  \\
&\lesssim\left|\psi_n\right|_{\dot{H}^1}\left\|\psi_n\right\|_{L^2}^3+\left|\psi_n\right|_{\dot{H}^1}^2\left\|\psi_n\right\|_{L^2}^4+\left|\psi_n\right|_{\dot{H}^1}^2\|\psi_n\|_{L^2}^2\\
&=\mathrm{o}\left(K^Q\left(\varphi_n\right)\right)
\end{aligned}
$$
for sufficiently large $n$. As a result, 
$$
K\left(\varphi_n\right)=K^Q\left(\varphi_n\right)-K^N\left(\varphi_n\right) \approx K^Q\left(\varphi_n\right)>0
$$
for sufficiently large $n$. This concludes the proof.
\end{proof}

Now we replace the functional $E$, which is unbounded from below, by a new positive functional $W$. At the same time, we enlarge the constraint set from the level surface $K=0$ (the ``mountain ridge'') to the sublevel set $K \leq 0$ (the ``mountain flank"). Let 
\begin{equation*}\label{functional H}
\begin{aligned}
W(\varphi) & :=E(\varphi)-\frac{1}{4} K(\varphi) \\
& =\frac{1}{4}\int_{\mathbb{R}}\left(\left| \varphi^{\prime}\right|^2-\omega|\varphi|^2+c \operatorname{Im}\left(\bar{\varphi} \varphi^{\prime}\right) -\frac{B}{3}|\varphi|^6\right)dx\\
&=\frac{1}{4}\int_{\mathbb{R}}\left(\left| (e^{icx/2} \varphi)^\prime\right|^2- \left( \omega + \frac{c^2}{4} \right) |\varphi|^2-\frac{B}{3}|\varphi|^6\right)dx.
\end{aligned}
\end{equation*}
Therefore, we have the strict monotonicity
\begin{equation}\label{monot. of W}
W\left(\lambda_1 \varphi\right)<W\left(\lambda_2 \varphi\right)
\end{equation}
for all $\varphi \in \widetilde{H}_c \backslash\{0\}$ and $0<\lambda_1<\lambda_2$. Then we can replace \eqref{E with K(varphi)=0} with 
\begin{equation}\label{E with K<=0}
\widetilde{I}_{0} :=\inf \left\{W(\varphi): \ \varphi \in \widetilde{H}_c \backslash\{0\} \quad \text{and} \quad K(\varphi) \leq 0  \right\}.
\end{equation}
This allows us to carry out the minimization over a broader admissible set while retaining equivalence to the original problem \eqref{E with K(varphi)=0}, as shown in Lemma \ref{equiv. lemma}.

\begin{lemma}\label{equiv. lemma}
The minimization problems \eqref{E with K(varphi)=0} and \eqref{E with K<=0} are equivalent in the following sense:
\begin{enumerate}[label=(\roman*)]
    \item $I_{0} = \widetilde{I}_{0} > 0$;
    \item minimizers of one problem are minimizers of the other.
\end{enumerate}
\end{lemma}
\begin{proof}
If $\varphi \in \widetilde{H}_c \backslash\{0\}$ satisfies $K(\varphi)=0$, then
$
W(\varphi)=E(\varphi). 
$
Hence,
$$
\begin{aligned}
I_{0} & =\inf \left\{E(\varphi): K(\varphi)=0\quad \text{with}\quad\varphi \in \widetilde{H}_c \backslash\{0\}\right\} \\
& =\inf \left\{W(\varphi): K(\varphi)=0\quad \text{with}\quad  \varphi \in \widetilde{H}_c \backslash\{0\}\right\} \\
& \geq \inf \left\{W(\varphi): K(\varphi) \leq 0\quad \text{with}\quad  \varphi \in \widetilde{H}_c \backslash\{0\}\right\} \\
& =\widetilde{I}_{0}.
\end{aligned}
$$
To prove the reverse inequality, let $\varphi \in \widetilde{H}_c \backslash\{0\}$ satisfy $K(\varphi)<0$. Then there exists $\lambda_0 \in(0,1)$ such that $K\left(\lambda_0 \varphi\right)=0$ by Lemmas \ref{lem:KQ-zero-limit} and \ref{lem:K-positive-eventually}. Moreover, we have
$$
E\left(\lambda_0 \varphi\right)=W\left(\lambda_0 \varphi\right)<W(\varphi)
$$
by \eqref{monot. of W}, which implies  $I_{0} \leq \widetilde{I}_{0}$. Therefore, $I_{0} = \widetilde{I}_{0}$.

It remains to prove the second statement. Let $\varphi$ be a minimizer for $\widetilde{I}_0$. If $K(\varphi)=0$, $\varphi$ is also a minimizer for $I_0$. We argue by contradiction. Suppose that $K(\varphi)<0$. Then, as mentioned above,  there exists $\lambda_0 \in(0,1)$, which is dependent on $\varphi$, such that
$$
K\left(\lambda_0 \varphi\right)=0.
$$
Thus,
$$
\widetilde{I}_0=W(\varphi)>W\left(\lambda_0 \varphi\right)=E\left(\lambda_0 \varphi\right) \geq I_0=\widetilde{I}_0
$$
by \eqref{monot. of W}. However, it contradicts the first statement, so $K(\varphi)<0$ is impossible.
Now let $\varphi$ be a minimizer for $I_0$.
Then 
$$
\widetilde{I}_0 \leq W(\varphi)=E(\varphi)=I_0=\widetilde{I}_0.
$$
Hence, $\varphi$ is also a minimizer for $\widetilde{I}_0$. This completes the proof.
\end{proof}
The equivalence established in Lemma \ref{equiv. lemma} allows us to work with \eqref{E with K<=0} from which the existence of a minimizer for  \eqref{E with K(varphi)=0} follows. Note that, in Theorem \ref{subcritical: B<0 and gamma>=0}, we have $\psi \in H^1(\mathbb{R};\mathbb{C})$. Therefore, for each fixed $t \in \mathbb{R}$, the function $u(\cdot,t)$ defined by \eqref{all transf. subcr.} belongs to $H^1(\mathbb{R};\mathbb{C})$ and solves \eqref{nonlocal DNLS}.

\begin{theorem}\label{subcritical: B<0 and gamma>=0}
The minimization problem \eqref{E with K(varphi)=0} admits a minimizer
$\varphi\in\widetilde H_c$ such that, upon writing
$$
\varphi(x)=e^{-icx/2}\psi(x),
$$
the function $\psi\in H^1(\mathbb{R};\mathbb{R})$ is nonnegative, even, and
nonincreasing on $[0,\infty)$. Moreover,
$$
\|\varphi\|_{\widetilde H_c}^2
=
\|\psi'\|_{L^2}^2
-
\left(\omega+\frac{c^2}{4}\right)\|\psi\|_{L^2}^2,
$$
and $\psi$ is a weak solution of
\begin{equation}\label{ODE 3}
-\psi''
-\left(\omega+\frac{c^2}{4}\right)\psi
+A|\psi|^2\psi
+B|\psi|^4\psi
+\gamma\psi\bigl(\mathcal{H}(|\psi|^2)\bigr)'
=0.
\end{equation}
\end{theorem}

\begin{proof}
Let $\left\{\varphi_n\right\}$ be a minimizing sequence for  \eqref{E with K<=0}.
By definition, there exists  $\left\{\psi_n\right\}$ in $H^1(\mathbb{R};\mathbb{C})$ such that
$$
\varphi_n=e^{-icx/2} \psi_n \quad \text { and } \quad\left\|\varphi_n\right\|_{\widetilde{H}_c}^2=\left|\psi_n\right|_{\dot{H}^1}^2-\left(\omega+\frac{c^2}{4}\right)\left\|\psi_n\right\|_{L^2}^2.
$$
Also let $\widetilde{\psi}_n :=\left|\psi_n\right|$ and $\widetilde{\varphi}_n:=e^{-icx/2}\widetilde{\psi}_n$. Then
\begin{equation*}\label{psi can be real-valued}
W(\varphi_n)\geq W(\widetilde{\varphi}_n) \quad \text{and} \quad K(\varphi_n)\geq K(\widetilde{\varphi}_n).
\end{equation*}
So, without loss of generality, we may assume that $\psi_n$ are nonnegative. We denote the Schwarz symmetrization of $\psi_n$ by $\psi_n^*$. For convenience, we set $\varphi_n^*=e^{-icx/2} \psi_n^*$. Then 
$$
\begin{aligned}
W(\varphi_n) & =\frac{1}{4}\int_{\mathbb{R}}\left(\left|(e^{icx/2} \varphi_n)^\prime\right|^2- \left( \omega + \frac{c^2}{4} \right) |\varphi_n|^2-\frac{B}{3}|\varphi_n|^6\right)dx \\
& =\frac{1}{4} \int_{\mathbb{R}}\left(\left| \psi_n^{\prime}\right|^2-\left(\omega+\frac{c^2}{4}\right)|\psi_n|^2-\frac{B}{3}|\psi_n|^6\right)dx \\
& \geq \frac{1}{4} \int_{\mathbb{R}}\left(\left| \psi_n^*\right|^2-\left(\omega+\frac{c^2}{4}\right)\left|\psi_n^*\right|^2-\frac{B}{3}\left|\psi^*\right|^6\right)dx \\
& =\frac{1}{4} \int_{\mathbb{R}}\left(\left|\left(e^{icx/2} \varphi_n^*\right)^\prime\right|^2-\left(\omega+\frac{c^2}{4}\right)\left|\varphi_n^*\right|^2-\frac{B}{3}\left|\varphi_n^*\right|^6\right)dx \\
& =W\left(\varphi_n^*\right)
\end{aligned}
$$
and
\begin{align*}
K(\varphi_n) & = \int_{\mathbb{R}}\bigg(
    \left|\varphi_n^{\prime}\right|^2
    - \omega |\varphi_n|^2
    + c  \operatorname{Im}\left(\bar{\varphi_n} \varphi_n^{\prime}\right)+  A |\varphi_n|^4+ B |\varphi_n|^6\\
    &\quad
    + \gamma |\varphi_n|^2\left(\mathcal{H}\left(|\varphi_n|^2\right)\right)^{\prime}
\bigg)dx \\
& = \int_{\mathbb{R}}\bigg(
    \left| \left(e^{icx/2} \varphi_n\right)^{\prime}\right|^2
    - \left( \omega + \frac{c^2}{4} \right) |\varphi_n|^2+  A |\varphi_n|^4+ B |\varphi_n|^6\\
    &\quad 
    + \gamma |\psi_n|^2\left(\mathcal{H}\left(|\psi_n|^2\right)\right)^{\prime}
\bigg)dx \\
& = \int_{\mathbb{R}}\bigg(
    \left| \psi_n^{\prime}\right|^2
    - \left( \omega + \frac{c^2}{4} \right) |\psi_n|^2
    +  A |\psi_n|^4+ B |\psi_n|^6\\
    & \quad
    + \gamma |\psi_n|^2\left(\mathcal{H}\left(|\psi_n|^2\right)\right)^{\prime}
\bigg)dx\\
& \geq \int_{\mathbb{R}}\bigg(
    \left|(\psi_n^*)^{\prime}\right|^2
    - \left( \omega + \frac{c^2}{4} \right) |\psi_n^*|^2
    +  A |\psi_n^*|^4+ B |\psi_n^*|^6\\
    &\quad
    + \gamma |\psi_n^*|^2\left(\mathcal{H}\left(|\psi_n^*|^2\right)\right)^{\prime}
\bigg)dx \\
& = \int_{\mathbb{R}}\bigg(
    \left|\left(e^{icx/2} \varphi_n^*\right)^{\prime}\right|^2
    - \left( \omega + \frac{c^2}{4} \right) |\varphi_n^*|^2
    +  A |\varphi_n^*|^4+ B |\varphi_n^*|^6\\
    &\quad 
    + \gamma |\psi_n^*|^2\left(\mathcal{H}\left(|\psi_n^*|^2\right)\right)^{\prime}
\bigg)dx \\
& = K(\varphi_n^*)
\end{align*}
by the rearrangement inequality (see e.g. \cite[Lemma 7.17]{Lieb2001}). Then by \cite[Theorem 3.5]{Lieb2001}, $\psi$ is also radially symmetric and nonincreasing about the origin of $\mathbb{R}$. So,  without loss of generality, we may assume that $\psi_n$ are radially symmetric and nonincreasing about the origin of $\mathbb{R}$. 

Since $\varphi_n$ is bounded in $\widetilde{H}_c$,  $\psi_n$ is bounded in $H^1(\mathbb{R}, \mathbb{R})$. By \cite[Lemma 2.4]{Miao2018}, upon passing to a suitable subsequence of $\psi_n$, we can assert that there exists $\psi \in H^1(\mathbb{R}, \mathbb{R})$ such that

$$
\begin{aligned}
\lim _{n \rightarrow \infty} \psi_n=\psi \quad& \text { weakly in } H^1(\mathbb{R}, \mathbb{R}), \\
\lim _{n \rightarrow \infty} \psi_n=\psi \quad& \text { strongly in } L^4(\mathbb{R}, \mathbb{R}) \text{ and } L^6(\mathbb{R}, \mathbb{R}).
\end{aligned}
$$
Since $\varphi_n=e^{-icx/2} \psi_n$ and $\varphi=e^{-icx/2} \psi$, it follows that
$$
\begin{aligned}
\lim _{n \rightarrow \infty} \varphi_n=\varphi \quad & \text { weakly in } \widetilde{H}_c, \\
\lim _{n \rightarrow \infty} \varphi_n=\varphi \quad & \text { strongly in } L^4(\mathbb{R};\mathbb{C})  \text{ and } L^6(\mathbb{R};\mathbb{C}).
\end{aligned}
$$
Hence,
$$ W(\varphi) \leq \lim _{n \rightarrow \infty} W\left(\varphi_n\right)=I_{0} $$
and
$$
K(\varphi) \leq \liminf _{n \rightarrow \infty} K\left(\varphi_n\right) \leq 0.
$$
It remains to prove that $\varphi \neq 0$. Suppose, for the sake of contradiction, that $\varphi=0$. Then we have
$$
\begin{aligned}
0 \leq \liminf _{n \rightarrow \infty} K^Q\left(\varphi_n\right) & =\liminf _{n \rightarrow \infty}\left(K\left(\varphi_n\right)+K^N\left(\varphi_n\right)\right) \\
& \leq \liminf _{n \rightarrow \infty} K\left(\varphi_n\right)+\lim _{n \rightarrow \infty} K^N\left(\varphi_n\right) \leq 0, 
\end{aligned}
$$
where we have used  \eqref{nonlocal bound}. By Lemma \ref{lem:K-positive-eventually}, we can find a subsequence $\varphi_{n_k}$ such that
$$
K\left(\varphi_{n_k}\right)>0 \quad \text { for } \text { sufficiently large } k.
$$
It is a contradiction with the choice of $\varphi_n$. Thus, $\varphi \neq 0$ and $\varphi$ is a minimizer for \eqref{E with K<=0}. By Lemma \ref{equiv. lemma}, $\varphi$ is also a minimizer for \eqref{E with K(varphi)=0}.

Since $ \varphi \in \widetilde{H}_c $ is a minimizer for \eqref{E with K(varphi)=0}, it follows from the Lagrange multiplier theorem that there exists  $ \mu \in \mathbb{R} $ such that
$$
\left\langle E'(\varphi), \nu \right\rangle = \mu \left\langle K'(\varphi), \nu \right\rangle \quad \text{for all } \nu \in \widetilde{H}_c.
$$
In particular, 
$$
\left\langle E'(\varphi), \varphi \right\rangle = \mu \left\langle K'(\varphi), \varphi \right\rangle.
$$
Then
$$
\begin{aligned}
0 =K(\varphi) & =\int_{\mathbb{R}}\bigg(\left| \varphi^{\prime}\right|^2+ A|\varphi|^4+B|\varphi|^6+ \gamma |\varphi|^2\left(\mathcal{H}\left(|\varphi|^2\right)^\prime\right)\bigg.-\omega|\varphi|^2\\
&\quad+c \operatorname{Im}\left(\bar{\varphi} \varphi^{\prime}\right)\bigg)dx \\
&= \mu\int_{\mathbb{R}}\bigg(2\left| \varphi^{\prime}\right|^2+4 A|\varphi|^4+6B|\varphi|^6 + 4\gamma |\varphi|^2\left(\mathcal{H}\left(|\varphi|^2\right)^\prime\right)\bigg.\bigg.-2\omega|\varphi|^2\\
&\quad+2c \operatorname{Im}\left(\bar{\varphi} \varphi^{\prime}\right)\bigg)dx \\
&=4\mu K(\varphi)-2 \mu \int_{\mathbb{R}}\left(\left| \varphi^{\prime}\right|^2-\omega|\varphi|^2+c \operatorname{Im}\left(\bar{\varphi} \varphi^{\prime}\right)\right)dx+2B\mu\int_{\mathbb{R}}  |\varphi|^6dx \\
&=-2 \mu\|\varphi\|_{\widetilde{H}_c}^2+2B \mu \|\varphi\|_{L^6}^6,
\end{aligned}
$$
which implies $ \mu = 0 $. Consequently,
$$
E'(\varphi) = 0 \quad \text{in } \widetilde{H}_c^*.
$$
Since $ \varphi(x) = e^{-icx/2} \psi(x) $, it follows that $ \psi$ satisfies \eqref{ODE 3}  in $ H^1(\mathbb{R};\mathbb{C})$.

\end{proof}

\section{Critical minimization on the Nehari manifold}\label{section 4}

Let $\omega, c \in \mathbb{R}$ satisfy $4 \omega=-c^2$. Also let  $ A>0$, $B< 0$ and $\gamma \geq 0$. Similarly to the subcritical case in Section \ref{section 3}, we are interested in 
\begin{equation}\label{E with K(varphi)=0 crit.}
I_{0}:=\inf \left\{E(\varphi):\ \varphi \in X_c \backslash\{0\} \quad \text{and} \quad K(\varphi)=0 \right\},
\end{equation}
where $E$ and $K$ are given by \eqref{Functional E} and \eqref{functional K}, respectively. 
We also define
$$
\begin{aligned}
K^Q(\varphi)
&:=
\int_{\mathbb{R}}
\left(
\big|(e^{icx/2}\varphi)'\big|^2
+A|\varphi|^4
\right)\,dx \\
&=
\int_{\mathbb{R}}
\left(
|\psi'|^2+A|\psi|^4
\right)\,dx,
\end{aligned}
$$
and
$$
K^N(\varphi):  =K^Q(\varphi)-K(\varphi)=\int_{\mathbb{R}}\left(-B|\varphi|^6-\gamma|\varphi|^2\left(\mathcal{H}\left(|\varphi|^2\right)\right)^{\prime}\right)dx.
$$
Then
$$
K^Q(\lambda \varphi)=\lambda^2 \int_{\mathbb{R}}\big|(e^{icx/2} \varphi)^\prime\big|^2dx+\lambda^4  A\int_{\mathbb{R}}|\varphi|^4dx
$$
for $\lambda>0$. Hence, we have proved the following lemma.
\begin{lemma}
Let $\varphi \in X_c \backslash\{0\}$. Then 
$$
\lim _{\lambda \rightarrow 0+} K^Q(\lambda \varphi)=0.
$$    
\end{lemma}

The next lemma demonstrates the behavior of $K$ about the origin of $X_c$.

\begin{lemma}
Let $\left\{\varphi_n\right\} $ be a bounded sequence in $X_c \backslash\{0\}$ such that
\begin{equation}\label{K Q vanishes}
\lim _{n \rightarrow \infty} K^Q\left(\varphi_n\right)=0.
\end{equation}
Then
$
K\left(\varphi_n\right)>0
$
for sufficiently large $n$.
\end{lemma}

\begin{proof}
Let $\left\{\varphi_n\right\} $ be a bounded sequence in $X_c \backslash\{0\}$ such that \eqref{K Q vanishes}. We estimate
$$
\begin{aligned}
K^Q\left(\varphi_n\right)
& = \int_{\mathbb{R}} \left( \big|(e^{icx/2} \varphi_n)^\prime\big|^2 +  A |\varphi_n|^4 \right)dx \\
& = \int_{\mathbb{R}} \left( \left| \psi_n^{\prime}\right|^2 +  A |\psi_n|^4 \right)dx\lesssim\left|\psi_n\right|_{\dot{H}^1}^2+\left\|\psi_n\right\|_{L^4}^{4}.
\end{aligned}
$$
Then by the Gagliardo-Nirenberg and \eqref{K Q vanishes}, 
$$
\begin{aligned}
|K^N(\varphi_n)|
&=
\left|
\int_{\mathbb{R}}
\left(
-B|\varphi_n|^6
-\gamma|\varphi_n|^2
\left(\mathcal{H}(|\varphi_n|^2)\right)'
\right)\,dx
\right|
\\
&\lesssim
|\psi_n|_{\dot H^1}^{2/3}
\|\psi_n\|_{L^4}^{16/3}
+
|\psi_n|_{\dot H^1}^{4/3}
\|\psi_n\|_{L^4}^{8/3}
\\
&\lesssim
\left(K^Q(\varphi_n)\right)^{5/3}
+
\left(K^Q(\varphi_n)\right)^{4/3}
\\
&=
o\left(K^Q(\varphi_n)\right)
\end{aligned}
$$
for sufficiently large $n$. 
As a result,
$$
K\left(\varphi_n\right)=K^Q\left(\varphi_n\right)-K^N\left(\varphi_n\right) \approx K^Q\left(\varphi_n\right)>0
$$
for sufficiently large $n$. This completes the proof.
\end{proof}

Now we replace the functional $E$, which is unbounded from below, by a new positive functional $W$. At the same time, we enlarge the constraint set from the level surface $K=0$ (the ``mountain ridge'') to the sublevel set $K \leq 0$ (the ``mountain flank"). Let 
\begin{equation*}
\begin{aligned}
W(\varphi) & :=E(\varphi)-\frac{1}{6} K(\varphi) \\
&=\frac{1}{3}\int_{\mathbb{R}}\left(\left| (e^{icx/2} \varphi)^\prime\right|^2+\frac{ A}{4}|\varphi|^4 +\frac{\gamma}{4} |\varphi|^2\left(\mathcal{H}\left(|\varphi|^2\right)\right)^{\prime}\right)dx.
\end{aligned}
\end{equation*}
Therefore,
\begin{equation*}
W\left(\lambda_1 \varphi\right)<W\left(\lambda_2 \varphi\right)
\end{equation*}
holds for all $\varphi \in X_c \backslash\{0\}$ and $0<\lambda_1<\lambda_2$. Then we can replace \eqref{E with K(varphi)=0 crit.} with 
\begin{equation}\label{E with K<=0 critical}
\widetilde{I}_{0} :=\inf \left\{W(\varphi):\ \varphi \in X_c \backslash\{0\} \quad \text{and}\quad  K(\varphi) \leq 0\right\}.
\end{equation}
This allows us to carry out the minimization over a broader admissible set while retaining equivalence to the original problem \eqref{E with K(varphi)=0 crit.}, as shown in Lemma \ref{equiv. lemma critic.}.

\begin{lemma}\label{equiv. lemma critic.}
The minimization problems \eqref{E with K(varphi)=0 crit.} and \eqref{E with K<=0 critical} are equivalent in the following sense:
\begin{enumerate}[label=(\roman*)]
    \item $I_{0} = \widetilde{I}_{0} > 0$;
    \item minimizers of one problem are minimizers of the other.
\end{enumerate}
\end{lemma}
\begin{proof}
The same argument as in Lemma \ref{equiv. lemma} applies.
\end{proof}

Note that, in Theorem \ref{main thm critical}, we have $\psi \in \dot{H}^1(\mathbb{R}; \mathbb{C})\cap L^4(\mathbb{R}; \mathbb{C})$. Therefore, for each fixed $t \in \mathbb{R}$, the function $u(\cdot,t)$ defined by \eqref{all transf. crit.} satisfies
$$
u(\cdot, t) \in \begin{cases}\dot{H}^1(\mathbb{R}; \mathbb{C})\cap L^4(\mathbb{R}; \mathbb{C}), \quad& c=0, \\ H_{\mathrm{loc}}^1 (\mathbb{R}; \mathbb{C})\cap L^4(\mathbb{R}; \mathbb{C}), & c \neq 0,\end{cases}
$$
and solves \eqref{nonlocal DNLS}.

\begin{theorem}\label{main thm critical}
The minimization problem \eqref{E with K(varphi)=0 crit.} admits a minimizer
$\varphi\in X_c$ such that, upon writing
$$
\varphi(x)=e^{-icx/2}\psi(x),
$$
the function
$$
\psi\in \dot H^1(\mathbb{R};\mathbb{R})\cap L^4(\mathbb{R};\mathbb{R})
$$
is nonnegative, even, and nonincreasing on $[0,\infty)$. Moreover,
$$
\|\varphi\|_{X_c}
=
\|\psi\|_{\dot H^1}
+
\|\psi\|_{L^4},
$$
and $\psi$ is a weak solution of
\begin{equation}\label{ODE 4}
-\psi''
+A|\psi|^2\psi
+B|\psi|^4\psi
+\gamma\psi\bigl(\mathcal{H}(|\psi|^2)\bigr)'
=0.
\end{equation}
\end{theorem}

\begin{proof}
One can show that \eqref{E with K(varphi)=0 crit.} admits at least one  minimizer, which is nonnegative, radially symmetric, and nonincreasing about the origin of $\mathbb{R}$,
by following the proof of Theorem \ref{subcritical: B<0 and gamma>=0}.

Since $ \varphi \in X_c $ is a minimizer for \eqref{E with K(varphi)=0 crit.}, it follows from the Lagrange multiplier theorem that there exists  $ \mu \in \mathbb{R} $ such that
$$
\left\langle E'(\varphi), \nu \right\rangle = \mu \left\langle K'(\varphi), \nu \right\rangle \quad \text{for all } \nu \in X_c.
$$
In particular, 
$$
\left\langle E'(\varphi), \varphi \right\rangle = \mu \left\langle K'(\varphi), \varphi \right\rangle.
$$
Then
$$
\begin{aligned}
0 =K(\varphi) & =\int_{\mathbb{R}}\bigg(\left| \varphi^{\prime}\right|^2+ A|\varphi|^4+B|\varphi|^6+ \gamma |\varphi|^2\left(\mathcal{H}\left(|\varphi|^2\right)\right)^{\prime}\bigg.\\
&-\omega|\varphi|^2+c \operatorname{Im}\left(\bar{\varphi} \varphi^{\prime}\right)\bigg)dx \\
&= \mu\int_{\mathbb{R}}\bigg(2\left| \varphi^{\prime}\right|^2+4 A|\varphi|^4+6B|\varphi|^6 + 4\gamma |\varphi|^2\left(\mathcal{H}\left(|\varphi|^2\right)\right)^\prime\bigg.\bigg.\\
&-2\omega|\varphi|^2+2c \operatorname{Im}\left(\bar{\varphi} \varphi^{\prime}\right)\bigg)dx \\
&=4\mu K(\varphi)-2 \mu \int_{\mathbb{R}}\left(\left| \varphi^{\prime}\right|^2-\omega|\varphi|^2+c \operatorname{Im}\left(\bar{\varphi} \varphi^{\prime}\right)\right)dx+2B\mu\int_{\mathbb{R}}  |\varphi|^6dx \\
&=-2 \mu|\varphi|_{\dot{H}_1}^2+2B \mu \|\varphi\|_{L^6}^6,
\end{aligned}
$$
which implies  $ \mu = 0 $. Consequently,
$$
E'(\varphi) = 0 \quad \text{in } X_c^*.
$$
Since $ \varphi(x) = e^{-icx/2} \psi(x) $, it follows that $ \psi $ solves  \eqref{ODE 4}
in  $\dot{H}^1(\mathbb{R}; \mathbb{C})\cap L^4(\mathbb{R}; \mathbb{C})$.
\end{proof}

\section{Subcritical Minimization under a Local--Nonlocal Quartic
Constraint}\label{section 5}

Let $\omega, c \in \mathbb{R}$ satisfy $\omega + c^2/4<0$. Also let $B=0$. We define the action functional
\begin{equation*}
\begin{aligned}
E(\varphi) :=& \int_{\mathbb{R}} \bigg(
 \frac{1}{2} \left| \varphi^{\prime}\right|^2
 - \frac{\omega}{2} |\varphi|^2
+ \frac{c}{2}   \operatorname{Im}\left(\bar{\varphi}   \varphi^{\prime}\right)
\bigg)  dx\\
=&\int_{\mathbb{R}} \Big(\frac{1}{2}|\psi^\prime|^2- \frac{1}{2}\left(\omega + \frac{c^2}{4}\right)|\psi|^2\Big)dx
\end{aligned}
\end{equation*}
and constraint
\begin{equation*}\label{eq:Q}
\begin{aligned}
Q(\varphi):=&\frac{1}{4}\int_{\mathbb{R}}\left( \alpha_1^2\left(|D|^{1/2}|\varphi|^2\right)^2 + \alpha_2^2|\varphi|^4\right)dx\\
=&\frac{1}{4}\int_{\mathbb{R}}\left( \alpha_1^2\left(|D|^{1/2}|\psi|^2\right)^2 + \alpha_2^2|\psi|^4\right)dx,
\end{aligned}
\end{equation*}
provided $\alpha_1^2+\alpha_2^2\neq 0$. 

In this section, we consider 
\begin{equation}\label{eq:Iq-def}
\begin{aligned}
I_q:&=\inf\{ E(\varphi):\ \varphi\in \widetilde{H}_c\quad \text{and}\quad Q(\varphi)=q \}\\
&=\inf\{ E(\psi):\ \psi\in H^1(\mathbb{R};\mathbb{C})\quad \text{and}\quad Q(\psi)=q \}.
\end{aligned}
\end{equation}
Let $\psi_n$ be a minimizing sequence for \eqref{eq:Iq-def}. It is apparent to see that
$$
\rho_n:=|\psi_n'|^2+|\psi_n|^2
$$
is bounded in $L^1(\mathbb{R};\mathbb{R})$. Then
$$
L:=\lim_{n\to\infty}\int_{\mathbb{R}}\rho_n dx,
$$
up to a subsequence. By normalizing $\rho_n$ and relabeling the resulting sequence, we may further assume that
$$
\int_{\mathbb{R}} \rho_n(x) dx = L 
\qquad \text{for all } n\in\mathbb{N}.
$$

We first present the existence theorem for the minimization problem \eqref{eq:Iq-def}. Its proof follows the concentration--compactness scheme: once the vanishing and dichotomy cases are ruled out, the compactness alternative yields a minimizer. We then obtain the associated Euler--Lagrange equation by using a Lagrange multiplier argument. The auxiliary results needed to exclude vanishing and dichotomy, including the scaling law and subadditivity properties, are established afterwards.  Note that $\psi\in H^1(\mathbb{R};\mathbb{C})$ in Theorem \ref{thm: existence}. Therefore, for each fixed $t\in\mathbb{R}$
the function $u(\cdot,t)$ defined by \eqref{all transf. subcr.} belongs to
$H^1(\mathbb{R};\mathbb{C})$ and solves \eqref{nonlocal DNLS}.

\begin{theorem}\label{thm: existence}
Let $\{\psi_n\}$ be a minimizing sequence for  \eqref{eq:Iq-def}, which can be taken  nonnegative. Then there exist   $\psi \in H^1(\mathbb{R};\mathbb{C})$  and $\{y_n\} \subset \mathbb{R}$  such that
$$
\psi_n( \cdot + y_n ) \to \psi
\quad \text{strongly in } H^1(\mathbb{R};\mathbb{C})
$$
up to a subsequence. Moreover, $\psi$ is a minimizer for \eqref{eq:Iq-def}. Hence, $\psi$ solves
\begin{equation}\label{ODE 5}
-\psi'' - \Bigl(\omega + \frac{c^2}{4}\Bigr)\psi+ A|\psi|^2\psi+\gamma\psi \left(\mathcal{H}(|\psi|^2)\right)^\prime=0
\quad \text{in }H^1(\mathbb{R};\mathbb{C}),
\end{equation}
where 
$\gamma=-\alpha_1^2\mu$, $ A=-\alpha_2^2\mu$ and  $\mu=\mu(\omega,c,q)>0$ is the Lagrange multiplier.
\end{theorem}
\begin{proof}
By Lemmas \ref{vanishing lemma} and \ref{dichotomy lemma}, the only possible alternative is the compactness. Let us show that  the compactness leads to the existence of a minimizer. By a change of variables,
$$
\int_{\left|x-y_n\right| \leq R(\varepsilon)} \rho_ndx \geq \int_{\mathbb{R}} \rho_ndx-\varepsilon.
$$
Therefore,
\begin{equation}\label{tail smallness}
\int_{|x-y_n|>R(\varepsilon)} \rho_n dx \leq \varepsilon \quad \text { for all } n.
\end{equation}
Since $\{\psi_n\}$ is bounded in $H^1(\mathbb{R};\mathbb{C})$,   there exists $\psi \in H^1(\mathbb{R};\mathbb{C})$ such that 
\begin{equation}\label{weak conv. of transl.}
\psi_n\left(\cdot+y_n\right) \rightharpoonup \psi \quad\text { in } H^1(\mathbb{R};\mathbb{C}) 
\end{equation}
and
\begin{equation}\label{local L2 conv}
\psi_n\left(\cdot+y_n\right) \rightarrow \psi \quad\text { in } L_{\text {loc}}^2(\mathbb{R};\mathbb{C})
\end{equation}
up to a subsequence.  Since $\psi\in H^1(\mathbb{R};\mathbb{C})$, we can choose $R$ large so that 
\begin{equation}\label{H1 tail bound}
\int_{|x|>R}\left( \left| \psi^{\prime}\right|^2+|\psi|^2\right)dx \leq \varepsilon.
\end{equation}
Combining \eqref{local L2 conv} with \eqref{H1 tail bound}, we have
\begin{equation*}
\psi_{n}\left(\cdot+y_n\right) \rightarrow \psi \quad\text { in } L^2(\mathbb{R};\mathbb{C}).
\end{equation*}
In particular,
\begin{equation}\label{L4 conv}
\psi_{n}\left(\cdot+y_n\right) \rightarrow \psi \quad\text { in } L^4(\mathbb{R};\mathbb{C}).
\end{equation}
By Lemma \ref{lem:split-global},
\begin{equation}\label{nonlocal convergence}
\left\| |D|^{1/2}\left(\left|\psi_n\left(\cdot+y_n\right)\right|^2-|\psi|^2\right)\right\|_{L^2} \rightarrow 0,
\end{equation}
so   $\psi$ is admissible.
In view of \eqref{weak conv. of transl.}, \eqref{L4 conv} and \eqref{nonlocal convergence},
$$
E(\psi) \leq \liminf _{n \rightarrow \infty} E\left(\psi_n\right)=I_q \quad \text{and}\quad Q(\psi_n)\to Q(\psi).
$$
Therefore, $\psi$ is a minimizer for \eqref{eq:Iq-def}.

The Lagrange multiplier theorem implies that  there exists   $ \mu \in \mathbb{R} $ such that
$$
\left\langle E'(\psi), \nu \right\rangle = \mu \left\langle Q'(\psi), \nu \right\rangle \quad \text{for all } \nu \in H^1(\mathbb{R};\mathbb{C}).
$$
In particular, 
$$
\left\langle E'(\psi), \psi \right\rangle = \mu \left\langle Q'(\psi), \psi \right\rangle.
$$
Thus,
$$
\mu = \frac{E(\psi)}{2Q(\psi)} > 0
$$
and $ \psi$ satisfies \eqref{ODE 5}  in $H^1(\mathbb{R};\mathbb{C})$ with $\gamma=-\alpha_1^2\mu$ and $ A=-\alpha_2^2\mu$.
\end{proof}

The remainder of this section is devoted to the auxiliary results used in the proof above. First, we establish a scaling law for $I_{q}$.

\begin{lemma}\label{scaling lemma}
Let $\theta, q > 0$. Then
$$
I_{\theta q} = q^{1/2} I_\theta.
$$
In particular, 
$$
I_q = q^{1/2} I_1.
$$
\end{lemma}

\begin{proof}
Let $\psi_q = q^{1/4} \psi$. Then
$$
E(\psi_q) = q^{1/2} E(\psi)
\quad \text{and} \quad
 Q(\psi_q) = q Q(\psi).
$$
Therefore,
$$
\begin{aligned}
I_{\theta q}
&= \inf \{ E(\psi_q) : \ \psi\in H^1(\mathbb{R};\mathbb{C})\enspace \text{and} \enspace Q(\psi_q) = \theta q \}\\
&= \inf \{ q^{1/2} E(\psi) :\ \psi\in H^1(\mathbb{R};\mathbb{C})\enspace \text{and}\enspace Q(\psi) = \theta \}\\
&= q^{1/2} I_\theta.
\end{aligned}
$$
\end{proof}

As a consequence, we have the strict subadditivity.

\begin{corollary}
Let $0 < q < 1$. Then
$$
I_1 < I_q + I_{1-q}.
$$    
\end{corollary}

Having established the strict subadditivity, we are now in a position to exclude the vanishing case.

\begin{lemma}\label{vanishing lemma}
Vanishing cannot occur.
\end{lemma}
\begin{proof}

Let us briefly sketch the proof. We argue by contradiction and assume that vanishing occurs.  Then
\begin{equation}\label{vanish}
\lim_{n\to\infty}\sup_{y\in\mathbb{R}}\int_{|x-y|\leq R}
\bigl(|\psi_n'(x)|^2+|\psi_n(x)|^2\bigr) dx=0
\end{equation}
for all $R>0$. So, we can assume that $R\ge 4$.
Using a covering of $\mathbb{R}$ by overlapping intervals and suitable cutoff functions, we split
$|D|^{1/2}(|\psi_n|^2)$ into a localized term and a tail term on each interval.
The localized term is controlled by a local $H^1$-estimate, and summing over all
intervals shows that its total contribution tends to zero by \eqref{vanish}. The tail term is small by the singular-integral representation of $|D|^{1/2}$ and
the separation of supports, and summing over all intervals shows that its total
contribution is also negligible. Hence,
$$
\bigl\||D|^{1/2}(|\psi_n|^2)\bigr\|_{L^2}\to 0,
$$
Using the standard approach, 
$$\|\psi_n\|_{L^4}\to 0.$$
Therefore, $$Q(\psi_n)\to 0,$$
which is a  contradiction.

Now we turn to the details. We set
$$
J_k:=(2k-1,2k+1),\qquad
J_k^*:=(2k-2,2k+2),\qquad
\mathcal{J}_k:=J_k^*\setminus J_k.
$$
Let $\chi_k\in C_c^\infty(\mathbb{R};\mathbb{R})$ satisfy
$$
\chi_k\equiv1\quad\text{on }J_k,
\qquad
\supp \chi_k\subset J_k^*,
$$
and with $\|\chi_k'\|_{L^\infty}\leq C_0$ for all $k$. Now we deal with the nonlocal term and local term separately.

\textbf{Step 1: The nonlocal term.} Let 
$$
T_{k,n}:=
\bigl\||D|^{1/2}((1-\chi_k)|\psi_n|^2)\bigr\|_{L^2(J_k)}.
$$
Then
$$
\bigl\||D|^{1/2}(|\psi_n|^2)\bigr\|_{L^2(J_k)}
\le
\bigl\||D|^{1/2}(\chi_k|\psi_n|^2)\bigr\|_{L^2}
+
T_{k,n}.
$$
Squaring and summing in $k$, we obtain
\begin{equation}\label{global split}
\bigl\||D|^{1/2}(|\psi_n|^2)\bigr\|_{L^2}^2
\lesssim
\sum_k \bigl\||D|^{1/2}(\chi_k|\psi_n|^2)\bigr\|_{L^2}^2
+
\sum_k T_{k,n}^2.
\end{equation}
We estimate the two terms on the right-hand side of \eqref{global split} separately.

\smallskip
\emph{Local contribution.}
Using $H^1(\mathbb{R};\mathbb{C})\hookrightarrow H^{1/2}(\mathbb{R};\mathbb{C})$,
$$
\bigl\||D|^{1/2}(\chi_k|\psi_n|^2)\bigr\|_{L^2}^2
\lesssim
\|\chi_k|\psi_n|^2\|_{H^1}^2
\lesssim
\|\psi_n\|_{H^1(J_k^*)}^4.
$$
Recall that $J_k^*$ has length $4$, so it is contained in an interval of the form $\{|x-y|\leq R\}$. Hence by \eqref{vanish},
$$
\|\psi_n\|_{H^1(J_k^*)}^2
=
\int_{J_k^*}\bigl(|\psi_n'|^2+|\psi_n|^2\bigr) dx
\leq \varepsilon
$$
for all $k$ and all sufficiently large $n$. Therefore
$$
\bigl\||D|^{1/2}(\chi_k|\psi_n|^2)\bigr\|_{L^2}^2
\lesssim
\varepsilon \|\psi_n\|_{H^1(J_k^*)}^2.
$$
Summing in $k$ and using that the intervals $J_k^*$ have bounded overlap, we get
\begin{equation}\label{local bound}
\sum_k \bigl\||D|^{1/2}(\chi_k|\psi_n|^2)\bigr\|_{L^2}^2
\lesssim
\varepsilon \sum_k \|\psi_n\|_{H^1(J_k^*)}^2
\lesssim
\varepsilon \|\psi_n\|_{H^1}^2.
\end{equation}

\smallskip
\emph{Tail contribution.}
We set
$$
f_n:=|\psi_n|^2,
\qquad
h_{k,n}:=(1-\chi_k)f_n.
$$
Since $h_{k,n}(x)=0$ for $x\in J_k$,  we have
$$
|D|^{1/2}h_{k,n}(x)
=
-c\, \mathrm{ p.v.}\int_{\mathbb{R}}\frac{(1-\chi_k(y))f_n(y)}{|x-y|^{3/2}} dy.
$$
Decomposing $\mathbb{R}=J_k\cup\mathcal{J}_k\cup(\mathbb{R}\setminus J_k^*)$, we split
$$
|D|^{1/2}h_{k,n}(x)=I_{k,n}^{\mathrm{ann}}(x)+I_{k,n}^{\mathrm{far}}(x),
$$
where
$$
I_{k,n}^{\mathrm{ann}}(x):=
-c\int_{\mathcal{J}_k}\frac{(1-\chi_k(y))f_n(y)}{|x-y|^{3/2}} dy,
\qquad
I_{k,n}^{\mathrm{far}}(x):=
-c\int_{\mathbb{R}\setminus J_k^*}\frac{f_n(y)}{|x-y|^{3/2}} dy.
$$
Thus,
$$
T_{k,n}
\le
T_{k,n}^{\mathrm{ann}}+T_{k,n}^{\mathrm{far}},
$$
where
$$
T_{k,n}^{\mathrm{ann}}:=\|I_{k,n}^{\mathrm{ann}}\|_{L^2(J_k)}\quad \text{and}
\quad
T_{k,n}^{\mathrm{far}}:=\|I_{k,n}^{\mathrm{far}}\|_{L^2(J_k)}.
$$
Now we will find upper estimates for the newly defined quantities.

\smallskip
\textit{(a) Far part.}
Since
$$
|x-y|^{-3/2}\lesssim |m|^{-3/2}\quad \text{for all }x\in J_k \text{ and } y\in J_{k+m} \text{ with }|m|\geq 2,
$$
and
$$|x-y|^{-3/2}\lesssim 1 \quad \text{for all } x\in J_k  \text{ and } y\in(\mathbb{R}\setminus J_k^*)\cap J_{k \pm1},$$
it follows that
$$
|I_{k,n}^{\mathrm{far}}(x)|
\lesssim
a_{k-1}+a_{k+1}+\sum_{|m|\ge2}|m|^{-3/2}a_{k+m},
$$
where $$
a_k:=\int_{J_k}f_n(y) dy=\int_{J_k}|\psi_n(y)|^2 dy.
$$
Let $w\in\ell^1(\mathbb{Z})$ be defined by
$$
w_0=0, \quad w_{\pm1}=1,\quad\text{and}\quad w_m=|m|^{-3/2} \enspace\,\text{ for all } |m|\geq 2.
$$
Then
$$
T_{k,n}^{\mathrm{far}}
\leq |J_k|^{1/2}\|I_{k,n}^{\mathrm{far}}\|_{L^\infty(J_k)}
\lesssim
(w*a)_k,
$$
where $(w * a)_k:=\sum_m w_m a_{k+m}$ is the discrete convolution. Hence, Young's inequality on $\mathbb{Z}$ yields
\begin{equation}\label{far bound rewrite}
\sum_k (T_{k,n}^{\mathrm{far}})^2
\lesssim
\|w*a\|_{\ell^2}^2
\lesssim
\|w\|_{\ell^1}^2\|a\|_{\ell^2}^2
\lesssim
\sum_k a_k^2.
\end{equation}
Now we estimate $\sum_k a_k^2$. Since each $J_k$ is contained in an interval of the form $\{|x-y|\leq R\}$, it follows from \eqref{vanish} that
$$
a_k
\leq
\int_{J_k}\bigl(|\psi_n'|^2+|\psi_n|^2\bigr) dx
\leq
\varepsilon
$$
for all $k$,
and therefore
$$
\sum_k a_k^2
\le
\sup_k a_k\sum_k a_k
\lesssim
\varepsilon \|\psi_n\|_{L^2}^2.
$$
Substituting this into \eqref{far bound rewrite}, we get
\begin{equation}\label{far final}
\sum_k (T_{k,n}^{\mathrm{far}})^2
\lesssim
\varepsilon \|\psi_n\|_{L^2}^2.
\end{equation}

\smallskip
\emph{(b) Annulus part.}
Using 
$$
|\chi_k(x)-\chi_k(y)|
\le
\|\chi_k'\|_{L^\infty}|x-y|
\le
C_0|x-y|.
$$
with $x\in J_k$ and $y\in \mathcal{J}_k$, and the Cauchy-Schwarz inequality,
$$
|I_{k,n}^{\mathrm{ann}}(x)|^2
\lesssim
\Bigl(\int_{\mathcal{J}_k}|x-y|^{-1} dy\Bigr)\|f_n\|_{L^2(\mathcal{J}_k)}^2.
$$
Integrating in $x\in J_k$ and using
$$
\int_{J_k}\int_{\mathcal{J}_k}|x-y|^{-1} dy dx<\infty,
$$
we obtain
$$
T_{k,n}^{\mathrm{ann}}
\lesssim
\|f_n\|_{L^2(\mathcal{J}_k)}
=
\|\psi_n\|_{L^4(\mathcal{J}_k)}^2.
$$
Therefore,
$$
\sum_k (T_{k,n}^{\mathrm{ann}})^2
\lesssim
\sum_k \|\psi_n\|_{L^4(\mathcal{J}_k)}^4.
$$
Since $\mathcal{J}_k$ is contained in an interval of the form $\{|x-y|\leq R\}$, it follows from \eqref{vanish} that
$$
\int_{\mathcal{J}_k}\bigl(|\psi_n'|^2+|\psi_n|^2\bigr) dx
\leq \varepsilon
$$
for all $k$.
By the Gagliardo--Nirenberg inequality,
\begin{equation*}
\|\psi_n\|_{L^4(\mathcal{J}_k)}^4
\lesssim
\|\psi_n\|_{L^2(\mathcal{J}_k)}^2
\Bigl(\|\psi_n'\|_{L^2(\mathcal{J}_k)}^2+\|\psi_n\|_{L^2(\mathcal{J}_k)}^2\Bigr)
\lesssim
\varepsilon \|\psi_n\|_{L^2(\mathcal{J}_k)}^2.
\end{equation*}
Summing in $k$ and keeping in mind bounded overlap of the annuli $\mathcal{J}_k$, we obtain
\begin{equation}\label{annulus bound rewrite}
\sum_k (T_{k,n}^{\mathrm{ann}})^2
\lesssim
\varepsilon \sum_k \|\psi_n\|_{L^2(\mathcal{J}_k)}^2
\lesssim
\varepsilon \|\psi_n\|_{L^2}^2.
\end{equation}
From \eqref{far final}--\eqref{annulus bound rewrite}, it follows that
\begin{equation}\label{Tk bound}
\sum_k T_{k,n}^2
\lesssim
\sum_k (T_{k,n}^{\mathrm{far}})^2+\sum_k (T_{k,n}^{\mathrm{ann}})^2
\lesssim
\varepsilon \|\psi_n\|_{L^2}^2.
\end{equation}

Combining \eqref{global split}, \eqref{local bound} and \eqref{Tk bound}, we have
\begin{equation}\label{nonlocal to zero}
\bigl\||D|^{1/2}(|\psi_n|^2)\bigr\|_{L^2}^2
\lesssim
\varepsilon \|\psi_n\|_{H^1}^2
\end{equation}
for sufficiently large $n$.

\medskip
\textbf{Step 2: The local quartic term.}
Now we will show that the local quartic term in the constraint is also arbitrarily small.
Applying the Gagliardo--Nirenberg inequality, 
$$
\|\psi_n\|_{L^4(J_k)}^4
\lesssim
\|\psi_n\|_{L^2(J_k)}^2
\Bigl(\|\psi_n'\|_{L^2(J_k)}^2+\|\psi_n\|_{L^2(J_k)}^2\Bigr).
$$
Summing over $k$, we get
\begin{equation}\label{quartic to zero}
\|\psi_n\|_{L^4}^4
\lesssim \|\psi_n\|_{L^2}^2
\sup_k \int_{J_k} (|\psi_n'|^2+|\psi_n|^2) dx
\lesssim\varepsilon\|\psi_n\|_{L^2}^2
\end{equation}
for sufficiently large $n$ by \eqref{vanish}. 

In view of \eqref{nonlocal to zero}--\eqref{quartic to zero},
$$Q(\psi_n)\to 0\quad \text{as }n\to \infty,$$
which leads to a contradiction.
\end{proof}

Then we exclude the dichotomy alternative.

\begin{lemma}\label{dichotomy lemma}
Dichotomy cannot occur.
\end{lemma}
\begin{proof}
By Lemma \ref{scaling lemma}, it suffices to consider the case $q=1$.
Thus, let $\{\psi_n\}$ be a minimizing sequence for $I_1$, so that
$$
Q(\psi_n)=1.
$$
Suppose that the dichotomy holds. The proof is divided into four steps. We first introduce the cutoff
decomposition, then estimate the interaction and remainder terms, identify
the limiting constraint values, and finally exclude all possible cases.

\textbf{Step 1: Cutoff decomposition.} 
Let $\xi_1,\xi_2\in C^\infty(\mathbb{R}, \mathbb{R})$ satisfy
$$
\operatorname{supp}\xi_1 \subset \{x \in \mathbb{R} : |x| \leq 2\}
$$
with $\xi_1(x) = 1$ for $|x| \leq 1$ and 
$$\operatorname{supp}\xi_2 \subset \{x \in \mathbb{R} : |x| \geq 1/2\}$$
with $\xi_2(x) = 1$ for $|x| \geq 1$.  
Using them, we also define
$$
\psi_{n,1}(x) := \xi_1\left(\frac{|x - y_n|}{R}\right)\psi_n(x),
$$
and
$$\psi_{n,2}(x) := \xi_2\left(\frac{|x - y_n|}{R_n}\right)\psi_n(x).$$
Observe that 
\begin{equation*}
\operatorname{supp} \psi_{n,1} \subset [-2R + y_n,  2R + y_n], 
\end{equation*}
with
$$\psi_{n,1} = \psi_n \ \text{on}\ [-R + y_n,  R + y_n],$$
and
\begin{equation*}
\operatorname{supp} \psi_{n,2} \subset 
\left(-\infty, -\frac{R_n}{2} + y_n\right] 
\cup 
\left[\frac{R_n}{2} + y_n,  +\infty\right),
\end{equation*}
with
$$\psi_{n,2} = \psi_n \ \text{on}\ 
\left(-\infty, -R_n + y_n\right) 
\cup 
\left(R_n + y_n,  +\infty\right).$$
Choosing $n$ sufficiently large so that 
$$
2R \leq \frac{R_n}{2},
$$
we ensure that
\begin{equation}\label{do not intersect}
\operatorname{supp} \psi_{n,1} \cap \operatorname{supp} \psi_{n,2} = \varnothing.
\end{equation}
We also define
$$
r_n := \psi_n - \psi_{n,1} - \psi_{n,2}.
$$
It is clear that
$$
\operatorname{supp} r_n 
\subset  \{x \in \mathbb{R} : R \leq |x - y_n| \leq R_n \}.
$$

\textbf{Step 2: Decomposition estimates.}
Expanding
\begin{equation*}
\int_{\mathbb{R}}(|D|^{1/2}|\psi_n|^2)^2dx
= \int_{\mathbb{R}} \big(|D|^{1/2} |\psi_{n,1} + \psi_{n,2} + r_n|^2\big)^2dx
\end{equation*} 
and grouping the terms according to their supports, we rewrite as follows
$$
\begin{aligned}
\int_{\mathbb{R}}(|D|^{1/2}|\psi_n|^2)^2dx
&= \int_{\mathbb{R}}(|D|^{1/2}|\psi_{n,1}|^2)dx + \int_{\mathbb{R}}(|D|^{1/2}|\psi_{n,2}|^2)dx\\
&\quad
+ \mathcal{I}_{12} + \mathcal{I}_{1r} + \mathcal{I}_{2r} + \mathcal{I}_r,
\end{aligned}
$$
where
\begin{itemize}
    \item $\mathcal{I}_{12}$ collects the interaction between $\psi_{n,1}$ and $\psi_{n,2}$,
    \item $\mathcal{I}_{ir}$ ($i=1,2$) collects the interaction between $\psi_{n,i}$ and $r_n$, and
    \item $\mathcal{I}_r$ contains all terms that are at least quadratic in $r_n$.
\end{itemize}  
Moreover, the cutoff construction yields
$$
\|r_n\|_{L^2}=O(\varepsilon^{1/2}), \qquad
|r_n|_{\dot H^1}=O(\varepsilon^{1/2}), \qquad
\|r_n\|_{L^\infty}=O(\varepsilon^{1/2}),
$$
and hence by the Sobolev interpolation and Kato–Ponce inequality,
$$
\| |D|^{1/2}\operatorname{Re}(\psi_{n,i}\bar{r_n})\|_{L^2}=O(\varepsilon^{1/2}), 
\qquad i=1,2.
$$
Consequently,
$$
\mathcal{I}_{1r}, \mathcal{I}_{2r}=O(\varepsilon^{1/2}).
$$
For the pure remainder part, we have
$$
\big\| |D|^{1/2}|r_n|^2\big\|_{L^2}=O(\varepsilon^{3/4}).
$$
Furthermore, the subleading mixed terms inside $\mathcal{I}_r$ are of smaller order, namely,
$$
O(\varepsilon^{5/4}),
$$
obtained by the Cauchy–Schwarz inequality. Finally, since the supports of $\psi_{n,1}$ and $\psi_{n,2}$ are separated by a distance $d_n \to \infty$,  
the integral representation \eqref{eq:|D|} of $|D|$ gives
$$
\int_{\mathbb{R}} |D|^{1/2}|\psi_{n,1}|^2   |D|^{1/2}|\psi_{n,2}|^2dx
= O(d_n^{-2}) = O(\varepsilon).
$$
Therefore,
\begin{equation}\label{eq:Q-final}
 Q(\psi_n)
= Q(\psi_{n,1}) + Q(\psi_{n,2}) + O(\varepsilon^{1/2}).
\end{equation}
Moreover,
\begin{equation}\label{energy decomposition}
E(\psi_n)
= E(\psi_{n,1}) + E(\psi_{n,2}) + O(\varepsilon).
\end{equation}

\textbf{Step 3: Limiting constraint values.} Since 
$$
Q(\psi_{n,i}) \lesssim \|\psi_{n,i}\|_{H^1}^2   \|\psi_{n,i}\|_{L^2}^2,
$$
$Q(\psi_{n,1})$ and $Q(\psi_{n,2})$ are  bounded. 
Passing to a subsequence if necessary, we define
$$
\lambda_i(\varepsilon) := \lim_{n \to \infty} Q(\psi_{n,i}), 
\qquad i = 1,2.
$$
Since $\lambda_1(\varepsilon)$ and $\lambda_2(\varepsilon)$ are bounded independently of $\varepsilon$, we can select a sequence $\varepsilon_j \rightarrow 0$ such that the limits $\lambda_i = \lim_{j \rightarrow \infty} \lambda_i(\varepsilon_j)$ exist. Then the decomposition 
\eqref{eq:Q-final} implies
$$
\begin{aligned}
1 = Q(\psi_n)
   &= \lim_{j \to \infty} Q(\psi_{n,1})
   + \lim_{j \to \infty} Q(\psi_{n,2})
   + \lim_{j \to \infty} O(\varepsilon_j^{1/2})
   \\
   &= \lambda_1 + \lambda_2.
\end{aligned}
$$

\textbf{Step 4: Exclusion of the two cases.}
We distinguish the following two cases:
\begin{enumerate}[label=\arabic*.]
    \item $\lambda_1 \in (0,1)$;
    \item $\lambda_1 = 0$ (hence $\lambda_2 = 1$).
\end{enumerate}

\textit{Case 1.}  
Assume $\lambda_1 \in (0,1)$. 
Using \eqref{energy decomposition}, we obtain
$$
\begin{aligned}
E(\psi_n)
&\geq I_{Q(\psi_{n,1})} + I_{Q(\psi_{n,2})} + O(\varepsilon_j)\\
&= \big((Q(\psi_{n,1}))^{1/2} + (Q(\psi_{n,2}))^{1/2}\big) I_1 + O(\varepsilon_j).
\end{aligned}
$$
We pass to the limit $n \to \infty$ to obtain
$$
I_1 \geq \bigl(\lambda_1(\varepsilon_j)^{1/2} + \lambda_2(\varepsilon_j)^{1/2}\bigr) I_1 + O(\varepsilon_j).
$$
Letting $j \to \infty$ yields
$$
\begin{aligned}
I_1 &\geq (\lambda_1^{1/2} + \lambda_2^{1/2}) I_1 
= I_{\lambda_1} + I_{\lambda_2} > I_{\lambda_1+\lambda_2} = I_1,
\end{aligned}
$$
which is a contradiction.

\textit{Case 2.}  
Assume $\lambda_1 = 0$.  
By the coercivity of $E$ and  dichotomy assumption, we have
$$
\begin{aligned}
E(\psi_{n,1})
&= \int_{\mathbb{R}} 
\left( \frac{1}{2} |\psi_{n,1}^\prime|^2 
- \frac{1}{2}\left(\omega + \frac{c^2}{4}\right) |\psi_{n,1}|^2 \right)dx \\
&\geq C_0 \int_{|x - y_n| \leq 2R} 
\left( |\psi_{n,1}^\prime|^2 + |\psi_{n,1}|^2 \right)dx 
+ O(\varepsilon_j) \\
&= C_0\bigl(\ell + O(\varepsilon_j)\bigr),
\end{aligned}
$$
where
$$
C_0 := \min\left\{ \frac{1}{2},   -\frac{1}{2}\left(\omega + \frac{c^2}{4}\right) \right\}.
$$
Hence,
$$
\begin{aligned}
E(\psi_n)
&= E(\psi_{n,1}) + E(\psi_{n,2}) + O(\varepsilon_j)\\
&\geq C_0\bigl(\ell + O(\varepsilon_j)\bigr)
   + Q(\psi_{n,2})^{1/2} I_1 + O(\varepsilon_j).
\end{aligned}
$$
Note that,
$$
\lim_{j \to \infty} \lim_{n \to \infty} Q(\psi_{n,2})
= \lim_{j \to \infty} \lambda_2(\varepsilon_j)
= \lambda_2 = 1,
$$
since $\lambda_1 = 0$.
Letting first $n \to \infty$ and then $j \to \infty$ gives
$$
I_1 \geq C_0 \ell + I_1 > I_1,
$$
which is a contradiction.    
Hence, the dichotomy cannot occur.
\end{proof}

The next lemma provides the key compactness statement for minimizing sequences.

\begin{lemma}\label{lem:split-global}
Let $\{\psi_n\}$ be a minimizing sequence for \eqref{eq:Iq-def}. Then
\begin{equation}\label{eq: nonlocal}
\big\| |D|^{1/2}\Big(|\psi_n(\cdot+y_n)|^2-|\psi|^2\Big)\big\|_{L^2}\to 0,
\end{equation}
up to a subsequence. 
\end{lemma}

\begin{proof}
Since we have ruled out dichotomy and vanishing,
we have uniform smallness of the tails \eqref{tail smallness} and the strong local $L^2$
-convergence \eqref{local L2 conv}. To prove \eqref{eq: nonlocal}, we decompose $$\big\| |D|^{1/2}\Big(|\psi_n(\cdot+y_n)|^2-|\psi|^2\Big)\big\|_{L^2}$$ into a local part and a tail part. More precisely, we localize it using a smooth cutoff function and treat separately the part supported on a ball of fixed radius and the remainder outside it.

For convenience, we set
$$
h_n:=|\psi_n(\cdot+y_n)|^2-|\psi|^2.
$$
Let $\chi\in C_c^\infty(-2,2)$ with $\chi\equiv1$ on $(-1,1)$. Let $R>R(\varepsilon)$ be sufficiently large. We define $\chi_R(x):=\chi\left(x/R\right)$ and write
$$
h_n=\chi_R^2 h_n+(1-\chi_R^2)h_n.
$$
We treat the local and tail parts separately
\begin{equation}\label{eq:local-split}
\| |D|^{1/2}h_n\|_{L^2}
 \leq  \| |D|^{1/2}(\chi_R^2 h_n)\|_{L^2}
 + \| |D|^{1/2}\big((1-\chi_R^2)h_n\big)\|_{L^2}.
\end{equation}

\emph{Local part.}
We also split $h_n$ as follows
$$
h_n=\bar{\psi}_n f_n+\psi \bar{f}_n,
$$
where $f_n:=\psi_n\left(\cdot+y_n\right)-\psi$.
Applying the Kato-Ponce inequality, we obtain 
\begin{align*}
\| |D|^{1/2}(\chi_R^2 h_n)\|_{L^2}
&\lesssim \big(\|\chi_R\psi_n\|_{L^\infty}+\|\chi_R\psi\|_{L^\infty}\big) \|\chi_R f_n\|_{H^{1/2}}\\
&+ \big(\|\chi_R\psi_n\|_{H^{1/2}}+\|\chi_R\psi\|_{H^{1/2}}\big) \|\chi_R f_n\|_{L^\infty}.
\end{align*}
Let $B_{2R}:=(-2R, 2R)$. Since
$$\|\chi_R f_n\|_{H^{1/2}}
\lesssim\left\|f_n\right\|_{L^2\left(B_{2 R}\right)}^{1/2}\|f_n\|_{H^{1}(B_{2R})}^{1/2}\to 0$$
by \eqref{local L2 conv}, and
$$
\left\|\chi_R f_n\right\|_{L^{\infty}}\lesssim \|f_n\|_{L^2(B_{2R})}^{1/2} \|f_n\|_{H^1(B_{2R})}^{1/2}\to 0,
$$
it follows that
\begin{equation}\label{eq:local-term-to0}
\| |D|^{1/2}(\chi_R^2 h_n)\|_{L^2}\to 0.
\end{equation}

\emph{Tail part}.
By the interpolation,
\begin{equation*}\label{tail part on a ball}
\begin{aligned}
\left\| |D|^{1/2}\bigl((1-\chi_R^2) h_n\bigr)\right\|_{L^2}
\lesssim
\| h_n\|_{L^2(B_{R}^c)}^{1/2}
\left\|(1-\chi_R^2) h_n\right\|_{H^1}^{1/2},
\end{aligned}
\end{equation*}
where $B_{R}^c:=\mathbb{R} \backslash B_{R}$.
We show that the first term on the right-hand side can be made arbitrarily small, while the second one remains bounded. We start with the first term.
$$
\begin{aligned}
\left\|h_n\right\|_{L^2(B_{R}^c)} &\leq\left(\left\|\psi_n\right\|_{L^{\infty}}+\|\psi\|_{L^{\infty}}\right)\left\|f_n\right\|_{L^2(B_{R}^c)}\\
&\lesssim \left(\left\|\psi_n\right\|_{H^1}+\|\psi\|_{H^1}\right)\left(\left\|\psi_n\right\|_{L^2(B_{R}^c)}+\|\psi\|_{L^2(B_{R}^c)}\right).
\end{aligned}
$$
By translation invariance and  \eqref{tail smallness}, 
$$
\sup_n\|\psi_n\|_{L^2(B_{R}^c)} \leq \varepsilon^{1/2}.
$$
Since $\psi\in L^2(\mathbb{R};\mathbb{C})$, we can choose sufficiently large $R$ such that
$$\left\|\psi\right\|_{L^2(B_{R}^c)}\leq \varepsilon^{1/2}.$$
Hence,
$$\left\|h_n\right\|_{L^2(B_{R}^c)}^{1/2} \lesssim \varepsilon^{1/4}.$$
Now we turn to the second term. Using that $H^1(\mathbb{R};\mathbb{C})$ is an algebra, we have
$$
\|(1-\chi_R^2)h_n\|_{H^1}\ \lesssim\ 1.
$$
Therefore, the tail part satisfies
\begin{equation}\label{eq:nonlocal-term-to0-fixed}
\big\| |D|^{1/2}\big((1-\chi_R^2)h_n\big)\big\|_{L^2}\ \lesssim\ \varepsilon^{1/4}.
\end{equation}
Finally, combining \eqref{eq:local-split}--\eqref{eq:nonlocal-term-to0-fixed}, we deduce
\eqref{eq: nonlocal}.
\end{proof}

\section{Critical Minimization under a Local Quartic Constraint}\label{section 6}

Let $\omega, c \in \mathbb{R}$ satisfy $\omega + c^2/4=0$. Also, let $B<0$ and $\gamma \geq 0$. We define the action functional
$$
\begin{aligned}
E(\varphi)
&:=\int_{\mathbb{R}}
\bigg(
\frac{1}{2}
\left|\left(e^{icx/2}\varphi\right)'\right|^2 +\frac{B}{6}|\varphi|^6
+\frac{\gamma}{4}
\left(|D|^{1/2}|\varphi|^2\right)^2
\bigg) dx\\
&=
\int_{\mathbb{R}}
\left(
\frac{1}{2}|\psi'|^2
+\frac{B}{6}|\psi|^6
+\frac{\gamma}{4}
\left(|D|^{1/2}|\psi|^2\right)^2
\right)dx
\end{aligned}
$$
and constraint
$$
Q(\varphi) := \frac{1}{4} \int_{\mathbb{R}}|\varphi|^4dx= \frac{1}{4} \int_{\mathbb{R}}|\psi|^4dx.
$$
  
Given $q>0$, we are interested in 
\begin{equation}\label{sextic+ nonlocal}
\begin{aligned}
I_q:&=\inf\{ E(\varphi):\ \varphi\in X_c\enspace \text{and}\enspace Q(\varphi)=q \}\\
&=\inf\{ E(\psi):\ \psi\in \dot{H}^1(\mathbb{R}; \mathbb{C})\cap L^4(\mathbb{R}; \mathbb{C})\enspace \text{and}\enspace Q(\psi)=q \}.
\end{aligned}
\end{equation}

We begin by deriving the scaling law for $I_q$.
\begin{lemma}
Let $\theta>0$ and $q>0$. Then
$$I_{\theta q} = q^{2}  I_{\theta}.$$
In particular,
$$
I_q = q^2 I_1.
$$
\end{lemma}
\begin{proof}
Let $\psi$ be a test function for $I_\theta$ with $\theta>0$. Also define $\psi_q(x) := \sqrt{q} \psi(qx)$ for $q>0$.
Then  
$$E\left(\psi_q\right)=q^2 E(\psi) \quad \text{and}\quad Q(\psi_q)=qQ(\psi)=\theta q.$$
Therefore,
$$
\begin{aligned}
 I_{\theta q}
&= \inf \{ E(\psi_q) :   \psi_q\in\dot{H}^1(\mathbb{R}; \mathbb{C})\cap L^4(\mathbb{R}; \mathbb{C}) \enspace \text{and} \enspace Q(\psi_q) = \theta q \}\\
&= \inf \{ q^{2} E(\psi) :  \psi\in\dot{H}^1(\mathbb{R}; \mathbb{C})\cap L^4(\mathbb{R}; \mathbb{C}) \enspace \text{and} \enspace Q(\psi) = \theta \}
\\
&= q^{2} I_\theta.
\end{aligned}
$$
Setting $\theta=1$ yields
$I_q = q^2 I_1$.
\end{proof}

The following lemma shows that $Q$ and the sextic term in $E$ are well-defined.

\begin{lemma}\label{lem: massfree}
Let $\psi\in \dot{H}^1(\mathbb{R}; \mathbb{C})\cap L^4(\mathbb{R}; \mathbb{C})$. Then we have
\begin{equation}\label{sextic term bound}
\int_{\mathbb{R}}|\psi|^6dx  \leq C  \|\psi\|_{L^4}^{16/3}  |\psi|_{\dot{H}^1}^{2/3}
\end{equation}
and 
\begin{equation}\label{nonlocal bound}
 \big\| |D|^{1/2}(|\psi|^2)\big\|_{L^2}^2  \leq C  \|\psi\|_{L^4}^{8/3}  |\psi|_{\dot{H}^1}^{4/3}.   
\end{equation}
\end{lemma}

\begin{proof}
We combine
$$
\int_{\mathbb{R}}|\psi|^6dx \leq \|\psi\|_{L^\infty}^{2}\|\psi\|_{L^4}^{4}
$$
with
$$
\|\psi\|_{L^\infty} \leq C  |\psi|_{\dot{H}^1}^{1/3} \|\psi\|_{L^4}^{2/3}
$$
to derive \eqref{sextic term bound}. For the nonlocal term, we combine the Kato-Ponce inequality
$$
\big\| |D|^{1/2}(|\psi|^2)\big\|_{L^2}
 \leq C 
\|\psi\|_{L^4} \big\| |D|^{1/2}\psi\big\|_{L^4}
$$
with the fractional Gagliardo--Nirenberg inequality 
$$
\big\| |D|^{1/2}\psi\big\|_{L^4}
 \leq C  |\psi|_{\dot{H}^1}^{2/3} \|\psi\|_{L^4}^{1/3}
$$
to obtain \eqref{nonlocal bound}.
\end{proof}
Then we show that the minimization problem \eqref{sextic+ nonlocal} is well posed.
\begin{corollary}
$-\infty<I_q<0$.
\end{corollary}

\begin{proof}
By Lemma \ref{lem: massfree},
$$
E(\psi)\ \ge\ \frac{1}{2} |\psi|_{\dot{H}^1}^2 - C q^{4/3} |\psi|_{\dot{H}^1}^{2/3}.
$$
Applying Young’s inequality to the  last   term, we derive an upper bound, 
$$
E(\psi)\ \geq  \Big(\frac{1}{2} - 2\varepsilon\Big) |\psi|_{\dot{H}^1}^2 - C_\varepsilon q^2
\ \ge\ -C_\varepsilon q^2,
$$
which implies $I_q>-\infty$. Under the scaling $\psi_\lambda(x)=\lambda^{1/4}\psi(\lambda x)$, we see that  $Q(\psi_\lambda)=q$ and
$$
E(\psi_\lambda) <0 \quad \text{for sufficiently small }\lambda>0.
$$
Therefore, $I_q<0$.
\end{proof}
The following lemma establishes the strict subadditivity.
\begin{lemma}
For every $q\in(0,1)$,
$$
I_1<I_q+I_{1-q}.
$$
\end{lemma}

\begin{proof}
Since
$$
q^2+(1-q)^2<1
\qquad\text{for } q\in(0,1),
$$
and $I_1<0$, it follows that
\begin{equation}\label{H1L4 subadditivity}
I_q+I_{1-q}>I_1.
\end{equation}
\end{proof}

\begin{lemma}\label{lem:bound-nonvanish.}
Let $\{\psi_n\}$ be a minimizing sequence for $I_q$. Then there exist $\Lambda=\Lambda(q)>0$ and $\delta>0$ such that
\begin{equation}\label{dot H1 bound}
|\psi_n|_{\dot{H}^1} \leq \Lambda
\quad\text{for all } n,
\end{equation}
and
\begin{equation}\label{lower bound L6}
\|\psi_n\|_{L^6} \geq \delta
\quad\text{for all sufficiently large } n.
\end{equation}
\end{lemma}

\begin{proof} 
Applying Lemma \ref{lem: massfree} and Young’s inequality, we obtain
$$
\begin{aligned}
|\psi_n|_{\dot{H}^1}^2
&\lesssim
E(\psi_n)+\|\psi_n\|_{L^6}^6\\
&\lesssim 
\sup_n E(\psi_n) + \varepsilon |\psi_n|_{\dot{H}^1}^2 + C_\varepsilon q^2.
\end{aligned}
$$
Choosing $\varepsilon>0$ sufficiently small, we have
$$
|\psi_n|_{\dot{H}^1}^2  \lesssim   \sup_n E(\psi_n)+q^2,
$$
which ensures \eqref{dot H1 bound}. 

We prove \eqref{lower bound L6} by contradiction.  
Assume that no constant $\delta>0$ satisfies the lower bound \eqref{lower bound L6}.  Then
$$
\liminf_{n\to\infty}\int_{\mathbb{R}}|\psi_n|^6 dx \leq 0.
$$
In this case, the negative term in  $E$ becomes negligible as $n\to\infty$. Consequently,
$$
I_q
  = \lim_{n\to\infty}E(\psi_n)
\geq \liminf_{n\to\infty}\left(\frac{B}{6}\int_{\mathbb{R}} |\psi_n|^6 dx\right)
  \geq 0.
$$
This contradicts $I_q<0$.  
Hence, there exists $\delta>0$ such that \eqref{lower bound L6} holds.
\end{proof}

\begin{lemma}
Let $\psi$ be a minimizer for \eqref{sextic+ nonlocal}. Then its symmetric decreasing
rearrangement $\psi^*$ is also a minimizer. Consequently, a minimizer
may be chosen nonnegative, even, and nonincreasing on $[0,\infty)$.
\end{lemma}

\begin{proof}
For any $\psi\in \dot{H}^1(\mathbb{R}; \mathbb{C})\cap L^4(\mathbb{R}; \mathbb{C})$, let $\psi^*$ denote the symmetric decreasing
rearrangement of $|\psi|$.
By the classical Pólya--Szegő inequality,
$$
\int_{\mathbb{R}} |{\psi^*}^{\prime}|^2dx
   \leq \int_{\mathbb{R}} | \psi^{\prime}|^2dx.
$$
Moreover, by the fractional Pólya--Szegő inequality (see e.g. \cite[Lemma 8.15]{Pava2009book}) and using the fact that $\left(|\psi|^2\right)^*=\left|\psi^*\right|^2$, we have
$$
\int_{\mathbb{R}}
   \big(|D|^{1/2}(|\psi^*|^2)\big)^2dx
   \le
   \int_{\mathbb{R}}
   \big(|D|^{1/2}(|\psi|^2)\big)^2dx.
$$
Therefore,
$$
E(\psi^*) \leq E(\psi)
\quad \text{and} \quad
Q(\psi^*) = Q(\psi).
$$
\end{proof}

Now let $\{\psi_n\}$ be a minimizing sequence for \eqref{sextic+ nonlocal}. Define
$$
\rho_n:=\frac14|\psi_n|^4.
$$
Then
$$
\int_{\mathbb R}\rho_n(x)\,dx=q
\qquad\text{for all }n\in\mathbb N.
$$
Thus, we may apply Lemma \ref{lem:CC} to the sequence
$\{\rho_n\}$.

Note that, in Theorem \ref{thm: fixed quartic term}, we have $\psi \in \dot{H}^1(\mathbb{R}; \mathbb{C})\cap L^4(\mathbb{R}; \mathbb{C})$. Therefore, for each fixed $t \in \mathbb{R}$, the function $u(\cdot,t)$ defined by \eqref{all transf. crit.} satisfies
$$
u(\cdot, t) \in \begin{cases}\dot{H}^1(\mathbb{R}; \mathbb{C})\cap L^4(\mathbb{R}; \mathbb{C}), \quad& c=0, \\ H_{\mathrm{loc}}^1(\mathbb{R}; \mathbb{C}) \cap L^4(\mathbb{R}; \mathbb{C}), & c \neq 0,\end{cases}
$$
and solves \eqref{nonlocal DNLS}.

\begin{theorem}\label{thm: fixed quartic term}
Let $\{\psi_n\}$ be a minimizing sequence for  \eqref{sextic+ nonlocal}. Then there exist     $\psi \in \dot{H}^1(\mathbb{R}; \mathbb{C})\cap L^4(\mathbb{R}; \mathbb{C})$  and $\{y_n\} \subset \mathbb{R}$  such that
$$
\psi_n(\cdot+y_n)\rightharpoonup\psi
\quad\text{in }\dot H^1(\mathbb R;\mathbb C),
\quad
\psi_n(\cdot+y_n)\to\psi
\quad\text{in }L^4(\mathbb R;\mathbb C),
$$
up to a subsequence. Moreover, $\psi$ is a minimizer for \eqref{sextic+ nonlocal}. Hence, $\psi$ solves
\begin{equation*}
-\psi^{\prime \prime}+ A|\psi|^2 \psi+B|\psi|^4 \psi+\gamma \psi\left(\mathcal{H}\left(|\psi|^2\right)\right)^{\prime}=0\qquad \text{in }\dot{H}^1(\mathbb{R}; \mathbb{C})\cap L^4(\mathbb{R}; \mathbb{C}),
\end{equation*} 
where
$$
\mu=\mu(B,\gamma,q)=\frac{2I_q}{q}<0
$$
is the Lagrange multiplier, and $A=-\mu$.
\end{theorem}

\begin{proof}

By Corollary \ref{cor: superaddit.} and Lemma \ref{lem:no-vanishing 1}, we rule out the vanishing and dichotomy cases.  Let us show that the compactness in Lemma \ref{lem:CC} leads to the existence of a minimizer for \eqref{sextic+ nonlocal}. By a change of variables,
$$
\int_{\left|x-y_n\right| \leq R(\varepsilon)} \rho_ndx \geq \int_{\mathbb{R}} \rho_ndx-\varepsilon.
$$
Therefore, 
\begin{equation*}
\int_{|x-y_n|>R(\varepsilon)} \rho_n dx\leq \varepsilon\qquad \text{for sufficiently large $n$.}
\end{equation*}
Since $\{\psi_n\}$ is bounded in $\dot{H}^1(\mathbb{R}; \mathbb{C})\cap L^4(\mathbb{R}; \mathbb{C})$ by \eqref{dot H1 bound},   there exists $\psi \in \dot{H}^1(\mathbb{R}; \mathbb{C})\cap L^4(\mathbb{R}; \mathbb{C})$ such that 
\begin{equation*}\label{weak convergence of translation 1}
\psi_n\left(\cdot+y_n\right) \rightharpoonup \psi \quad\text { in } \dot{H}^1(\mathbb{R}; \mathbb{C})\cap L^4(\mathbb{R}; \mathbb{C})
\end{equation*}
up to a subsequence.
Moreover,  we can choose $R$ large so that 
\begin{equation}\label{psi 4 tail bound}
\int_{|x|>R} |\psi|^4 dx \leq \varepsilon.
\end{equation}
Combining 
\begin{equation*}
\psi_n\left(\cdot+y_n\right) \rightarrow \psi \quad\text { in } L_{\text {loc}}^4(\mathbb{R};\mathbb{C})
\end{equation*}
 with \eqref{psi 4 tail bound}, we have 
\begin{equation}\label{L4 strong conv}
\psi_n\left(\cdot+y_n\right) \rightarrow \psi \quad\text { in } L^4(\mathbb{R};\mathbb{C}).
\end{equation}
From  \eqref{sextic term bound} and \eqref{L4 strong conv}, we deduce
\begin{equation}\label{sextic conv.}
\psi_n\left(\cdot+y_n\right) \rightarrow \psi \quad \text { in } L^6(\mathbb{R};\mathbb{C}).
\end{equation}
Now we show
\begin{equation}\label{nonlocal conv.}
\big| |\psi_n(\cdot+y_n)|^2 - |\psi|^2 \big|_{\dot{H}^{1/2}}
 \to  0.
\end{equation}
Let $ h_n := |\psi_n(\cdot+y_n)|^2 - |\psi|^2 $.  
By the interpolation, we have
\begin{equation*}\label{eq:interpolation}
|h_n|_{\dot{H}^{1/2}}^4
 \leq C
\|h_n\|_{L^2}^2 |h_n|_{\dot{H}^1}^2.
\end{equation*}
It suffices to show that $\|h_n\|_{L^2} \to 0$ and that $|h_n|_{\dot{H}^1}$ is  bounded. Indeed,
$$
\begin{aligned}
\|h_n\|_{L^2}\leq \|\psi_n(\cdot+y_n) - \psi\|_{L^4} (\|\psi_n(\cdot+y_n)\|_{L^4} + \|\psi\|_{L^4})
\to 0
\end{aligned}
$$
by \eqref{L4 strong conv}.
Since
\begin{equation*}\label{h_n bound}
|h_n|_{\dot{H}^1}^2 
\leq 8\left(
\|\psi_n(\cdot+y_n)\|_{L^{\infty}}^2
\left|\psi_n(\cdot+y_n)\right|_{\dot{H}^1}^2
+\|\psi\|_{L^{\infty}}^2
|\psi|_{\dot{H}^1}^2
\right)
\end{equation*}
we see that $|h_n|_{\dot{H}^1}^2$ is  bounded, so we have \eqref{nonlocal conv.}.

In view of \eqref{L4 strong conv}--\eqref{nonlocal conv.}, 
we see that 
$$
E(\psi) \leq \liminf _{n \rightarrow \infty} E\left(\psi_n\right)=I_q \quad \text{and}\quad Q(\psi_n)\to Q(\psi).
$$
Therefore, $\psi$ is a minimizer for \eqref{sextic+ nonlocal}. Then the Lagrange multiplier theorem asserts that  there exists a real number $ \mu \in \mathbb{R} $ such that
$$
\left\langle E'(\psi), \nu \right\rangle = \mu \left\langle Q'(\psi), \nu \right\rangle \quad \text{for all } \nu \in \dot{H}^1(\mathbb{R}; \mathbb{C})\cap L^4(\mathbb{R}; \mathbb{C}).
$$

Finally, since
$$
E(t\psi)\ge I_{t^4q}=t^8I_q,
\qquad Q(t\psi)=t^4q,
$$
with equality at $t=1$, we obtain
$$
8I_q=\langle E'(\psi),\psi\rangle
=\mu\langle Q'(\psi),\psi\rangle
=4\mu q.
$$
Thus,
$$
\mu=\frac{2I_q}{q}<0.
$$
\end{proof}

The remainder of this section is dedicated to the auxiliary lemmas used in the proof. The next step is to rule out the dichotomy alternative.

\begin{lemma}\label{lem:conc.-H1L4}
Let $\Lambda_1:=\max\{\Lambda,(4q)^{1/4}\}$ and $\delta>0$. Then there exists a constant 
$\eta=\eta(\Lambda_1,\delta)>0$ such that if 
$f \in \dot{H}^1(\mathbb{R}; \mathbb{C})\cap L^4(\mathbb{R}; \mathbb{C})$ with
$$
|f|_{\dot{H}^1} + \|f\|_{L^4} \leq \Lambda_1
\quad  \text{and} \quad
\|f\|_{L^6} \geq \delta,
$$
then
$$
\sup_{y\in\mathbb{R}}
\int_{y-1/2}^{ y+1/2} |f|^6dx
\;\ge\; \eta.
$$
\end{lemma}

\begin{proof}
Let
$J_k:=(k-1/2,k+1/2)$, $k\in\mathbb{Z}$. Summing over $k$, we have
\begin{equation*}
\sum_{k\in\mathbb{Z}} \int_{J_k} |f|^6dx
= \|f\|_{L^6}^6 \geq \delta^6,
\end{equation*}
while
\begin{equation}\label{eq:H1L4-decomp}
\sum_{k\in\mathbb{Z}} \int_{J_k} (|f'|^2 + |f|^4)dx
= |f|_{\dot{H}^1}^2 + \|f\|_{L^4}^4
\leq \Lambda_1^2 + \Lambda_1^4.
\end{equation}
Also, let 
$$R_f:=\frac{\Lambda_1^2+\Lambda_1^4}{\|f\|_{L^6}^6}.$$ 
If each interval $J_k$ satisfied
$$
\int_{J_k} (|f'|^2+|f|^4)dx
> 
R_f
\int_{J_k} |f|^6dx,
$$
then summing over $k$ would contradict \eqref{eq:H1L4-decomp}.  
Thus, there exists $k_0\in\mathbb{Z}$ such that
\begin{equation}\label{eq: pigeonhole}
\int_{J_{k_0}} (|f'|^2+|f|^4)dx
\leq 
R_f \int_{J_{k_0}} |f|^6dx.
\end{equation}
From \eqref{eq: pigeonhole}, we have
\begin{equation}\label{eq:dotH1-L6 and L4-L6}
\|f\|_{\dot{H}^1(J_{k_0})} \leq R_f^{1/2}\|f\|_{L^6(J_{k_0})}^3\quad \text{and} 
\quad \|f\|_{L^4(J_{k_0})} \leq R_f^{1/4}\|f\|_{L^6(J_{k_0})}^{3/2}.
\end{equation}
Moreover, by the Gagliardo--Nirenberg inequality, 
\begin{equation}\label{eq:GN-local}
\|f\|_{L^6(J_{k_0})}
\leq C \left( \|f\|_{\dot{H}^1(J_{k_0})}^{1/9} \|f\|_{L^4(J_{k_0})}^{8/9} +  \|f\|_{L^4(J_{k_0})}\right).
\end{equation}
Combining \eqref{eq:dotH1-L6 and L4-L6} with \eqref{eq:GN-local} and dividing by $\|f\|_{L^6(J_{k_0})}$ yields
\begin{equation}\label{eq:key}
1
\leq C R_f^{5/18} \|f\|_{L^6(J_{k_0})}^{2/3} \;+\; C R_f^{1/4} \|f\|_{L^6(J_{k_0})}^{1/2}.
\end{equation}
At least one term on the right-hand side of \eqref{eq:key} must be $\geq 1/2$.  
This leads to two possibilities. If $C R_f^{1/4} \|f\|_{L^6(J_{k_0})}^{1/2}\geq 1/2$, then  
$$
\|f\|_{L^6(J_{k_0})}^6 \geq\left(\frac{1}{2 C}\right)^{12} R_f^{-3}.
$$
If $C R_f^{5/18} \|f\|_{L^6(J_{k_0})}^{2/3}\geq 1/2$, then  
$$
\|f\|_{L^6(J_{k_0})}^6 \geq\left(\frac{1}{2 C}\right)^9 R_f^{-5 / 2}.
$$
Since $\|f\|_{L^6}\geq \delta$,  it follows that
$$
R_f
\leq \frac{\Lambda_1^2+\Lambda_1^4}{\delta^6}.
$$
Substituting this bound into the two inequalities above gives
$$
\begin{aligned}
\|f\|_{L^6(J_{k_0})}^6
\;&\ge\;
\min\left\{
\left(\frac{1}{2 C}\right)^{12} \frac{\delta^{18}}{(\Lambda_1^2+\Lambda_1^4)^3},
\;
\left(\frac{1}{2 C}\right)^9 \frac{\delta^{15}}{(\Lambda_1^2+\Lambda_1^4)^{5/2}}
\right\}
\\
&=: \eta(\Lambda_1,\delta)>0.
\end{aligned}
$$
Finally, the interval $J_{k_0}$ is of the form 
$(y-1/2,y+1/2)$ for some $y\in\mathbb{R}$, which completes the proof.
\end{proof}

Given $r>0$, we introduce
$$
M_n(r) := \frac{1}{4}\sup_{y\in\mathbb{R}} \int_{y-r}^{y+r} |\psi_n|^4dx\quad \text{and}\quad M(r) := \lim_{n\to\infty} M_n(r).
$$
Then $M(r)$ is nondecreasing in $r$ and satisfies $M(r)\leq q$. Also, let
$$
\alpha:=\lim _{r \rightarrow \infty} M(r).
$$
It is clear that $0\leq \alpha \leq q$.

The following lemma is used to analyze the behavior of minimizing sequences in the case $0< \alpha<q$.

\begin{lemma}\label{lem:split.}
For every $\varepsilon>0$, there exist  $n(\varepsilon)\in\mathbb{N}$ and 
sequences $\{g_n,g_{n+1},\dots\}$ and $\{h_n,h_{n+1},\dots\}$ in  $\dot{H}^1(\mathbb{R}; \mathbb{C})\cap L^4(\mathbb{R}; \mathbb{C})$ 
 such that for all $n\geq n(\varepsilon)$,
\begin{equation*}\label{stat.1}
  |Q(g_n)-\alpha|<\varepsilon,  
\end{equation*}
\begin{equation*}\label{stat.2}
 |Q(h_n)-(q-\alpha)|<\varepsilon,   
\end{equation*}
and
\begin{equation}\label{stat.3}
E(\psi_n)\geq E(g_n)+E(h_n)-\varepsilon.
\end{equation}
\end{lemma}

\begin{proof}
We choose $\chi \in C_0^\infty(-2,2)$ such that $\chi \equiv 1$ on $[-1,1]$, 
and let $\eta \in C^\infty(\mathbb{R})$ satisfy
$\chi^2 + \eta^2 \equiv 1$ on $\mathbb{R}$. 
Given $r>0$, we introduce the rescaled cutoff functions
$$
\chi_r(x) := \chi \left(x/r\right) \quad\text{and}\quad \eta_r(x) := \eta\left(x/r\right).
$$
Note that
$$
\alpha - \varepsilon < M(r) \leq M(2r) \leq \alpha
$$
for all sufficiently large  $r$. Suppose that a suitable value of $r$ has been selected.
Then there exists $N \in \mathbb{N}$ sufficiently large such that
$$
\alpha-\varepsilon<M_n(r) \leq M_n(2 r)<\alpha+\varepsilon
$$
for all $n \geq N$. Hence, for each $ n \geq N $, there exists $ y_n \in \mathbb{R} $ such that 
\begin{equation}\label{eq:mass-left}
\int_{y_n - r}^{y_n + r} |\psi_n|^4dx > \alpha - \varepsilon,
\end{equation}
and
\begin{equation}\label{eq:mass-right}
\int_{y_n - 2r}^{y_n + 2r} |\psi_n|^4dx < \alpha + \varepsilon.
\end{equation}
Let us also define
$$
g_n := \chi_r(\cdot - y_n)\psi_n
\quad \text{and}\quad
h_n := \eta_r(\cdot - y_n)\psi_n.
$$
Then $g_n$ and $h_n$ satisfy 
$$|Q(g_n)-\alpha|<\varepsilon \quad \text{ and } \quad|Q(h_n)-(q-\alpha)|<\varepsilon.$$
It remains to prove \eqref{stat.3}.  For simplicity, we  denote by $\chi_r$ and $\eta_r$ the translated cutoff functions 
$\chi_r(x - y_n)$ and $\eta_r(x - y_n)$.  Note that,
\begin{equation}\label{split identity}
\begin{aligned}
E(g_n)+E(h_n)
&= \frac{1}{2}\int_{\mathbb{R}}(|g_n'|^2+|h_n'|^2)dx
  + \frac{B}{6}\int_{\mathbb{R}}(|g_n|^6+|h_n|^6)dx
 \\
 &+ \frac{\gamma}{4}\int_{\mathbb{R}}\Big(|g_n|^2
      |D|\left(|g_n|^2\right)
     +|h_n|^2|D|\left(|h_n|^2\right)
    \Big)dx \\
&= \frac{1}{2}\left(
   \int_{\mathbb{R}}\chi_r^2|\psi_n'|^2dx
   +2\int_{\mathbb{R}}\chi_r\chi_r'\psi_n\psi_n'dx
   +\int_{\mathbb{R}}(\chi_r')^2\psi_n^2dx \right.\\
& \left.
   +\int_{\mathbb{R}}\eta_r^2|\psi_n'|^2dx
+2\int_{\mathbb{R}}\eta_r\eta_r'\psi_n\psi_n'dx
   +\int_{\mathbb{R}}(\eta_r')^2\psi_n^2dx
   \right)\\
& +\frac{B}{6}\left(
   \int_{\mathbb{R}}\chi_r^2|\psi_n|^6dx
   +\int_{\mathbb{R}}\eta_r^2|\psi_n|^6dx
   \right)\\
& -\frac{B}{6}\int_{\mathbb{R}}
   \big[(\chi_r^2-\chi_r^6)+(\eta_r^2-\eta_r^6)\big]|\psi_n|^6dx\\
&+\frac{\gamma}{4}\left(
   \int_{\mathbb{R}}\chi_r^2|\psi_n|^2|D|(\chi_r^2|\psi_n|^2)dx
   +\int_{\mathbb{R}}\eta_r^2|\psi_n|^2|D|(\eta_r^2|\psi_n|^2)dx
   \right).
\end{aligned}
\end{equation}
Since  $(\chi_r^2+\eta_r^2)' = 0$,
it follows that
\begin{equation}\label{cross terms vanish}
2\int_{\mathbb{R}}\chi_r\chi_r'\psi_n\psi_n'dx
+ 2\int_{\mathbb{R}}\eta_r\eta_r'\psi_n\psi_n'dx = 0.
\end{equation}
Using 
$$
\left|\chi_r^{\prime}\right|+\left|\eta_r^{\prime}\right| \lesssim r^{-1}
$$
and the fact that both derivatives are supported 
in the annulus $\Omega_r:=\{x\in\mathbb{R}:r<|x-y_n|<2r\}$, we obtain
$$
\int_{\mathbb{R}}\big((\chi_r')^2+(\eta_r')^2\big)\psi_n^2dx
\lesssim r^{-2}\int_{\Omega_r}|\psi_n|^2dx \lesssim  r^{-3/2}\|\psi_n\|_{L^4}^2=O(r^{-3/2}).
$$
Since  
$\chi_r^2-\chi_r^6$ and $\eta_r^2-\eta_r^6$ are 
supported  in $\Omega_r$,  it follows that
\begin{equation}\label{2 6 chi eta}
\left|
\int_{\mathbb{R}}
   \Big((\chi_r^2-\chi_r^6)+(\eta_r^2-\eta_r^6)\Big)
   |\psi_n|^6dx
\right|
\leq \|\psi_n\|_{L^{\infty}}^2 
\int_{\Omega_r} |\psi_n(x)|^4dx
\lesssim \varepsilon
\end{equation}
by \eqref{eq:mass-left} and \eqref{eq:mass-right}.

If we show that $$\left|\int_{\mathbb{R}}\chi_r^2|\psi|^2|D|\left(\chi_r^2|\psi|^2\right)dx+\int_{\mathbb{R}}\eta_r^2|\psi|^2|D|\left(\eta_r^2|\psi|^2\right)dx-\int_{\mathbb{R}}|\psi|^2|D|\left(|\psi|^2\right)dx\right|<\varepsilon,$$
then combining \eqref{split identity}--\eqref{2 6 chi eta}, we derive \eqref{stat.3}. By using the commutator identity,
\begin{equation*}
\begin{aligned}
\int_{\mathbb{R}}\chi_r^2|\psi_n|^2|D|(\chi_r^2|\psi_n|^2)dx =\int_{\mathbb{R}}\chi_r^2|\psi_n|^2\left[|D|,\chi_r^2\right]|\psi_n|^2dx+\int_{\mathbb{R}}\chi_r^4|\psi_n|^2|D||\psi_n|^2dx.
\end{aligned}
\end{equation*}
Then we find an upper bound for the first term
\begin{equation*}
\begin{aligned}
\int_{\mathbb{R}}\chi_r^2|\psi_n|^2\left[|D|,\chi_r^2\right]|\psi_n|^2dx&\leq \left\|\chi_r^2|\psi_n|^2\right\|_{L^2}\left\|\left[|D|,\chi_r^2\right]|\psi_n|^2\right\|_{L^2}\\
&\leq C\|\psi_n\|_{L^4}^2\left\|(\chi_r^2)^{\prime}\right\|_{L^{\infty}}\|\psi_n\|_{L^4}^2\\
&\leq C r^{-1}\left\|\psi_n\right\|_{L^4}^4
\end{aligned}
\end{equation*}
by using H\"older's inequality and Coifman-Meyer estimate \cite[Lemma 6.7]{Albert1999}. One can show that
\begin{equation*}
\int_{\mathbb{R}}\eta_r^2|\psi_n|^2\left[|D|,\eta_r^2\right]|\psi_n|^2dx=O(r^{-1}).
\end{equation*}
Thus,
$$
\begin{aligned}
&\left|\int_{\mathbb{R}} \Big(\chi_r^2|\psi_n|^2|D|\left(\chi_r^2|\psi_n|^2\right)
+\eta_r^2|\psi_n|^2|D|\left(\eta_r^2|\psi_n|^2\right)
-|\psi_n|^2|D|\left(|\psi_n|^2\right)\Big)dx\right|\\
&=\left|\int_{\mathbb{R}} \left(\chi_r^4+\eta_r^4-\chi_r^2-\eta_r^2\right)|\psi_n|^2 |D|\left(|\psi_n|^2\right)dx\right|
+O(r^{-1})
\\
&\leq 2\int_{\mathbb{R}} \chi_r^2\eta_r^2|\psi_n|^2|D|\left(|\psi_n|^2\right)dx+O(r^{-1})
\\
&\leq 2\int_{\mathbb{R}} \mathds{1}_{\Omega_r}|\psi_n|^2|D|\left(|\psi_n|^2\right)dx+O(r^{-1})<\varepsilon,
\end{aligned}
$$
where we have chosen $r$ sufficiently large so that the term $O(r^{-1})$ 
is smaller than $\varepsilon$ in absolute value.

\end{proof}

\begin{corollary}\label{cor: superaddit.}
If $0<\alpha<q$, then
$$
I_q  \geq  I_\alpha + I_{q-\alpha}.
$$
\end{corollary}

\begin{proof}
Let $\{\psi_n\}$ be a minimizing sequence for $I_q$,
and let $\{g_n\}$ and $\{h_n\}$ be the sequences defined in
Lemma \ref{lem:split.}.
Then for every $\varepsilon>0$ and all sufficiently large $n$,
$$
|Q(g_n)-\alpha| < \varepsilon, 
$$
$$|Q(h_n)-(q-\alpha)| < \varepsilon, $$
and
$$E(\psi_n) \geq E(g_n)+E(h_n)-\varepsilon.$$
For such $n$, we define
$$
\beta_g := \Bigl(\frac{\alpha}{Q(g_n)}\Bigr)^{1/4}\quad \text{and}\quad \beta_h := \Bigl(\frac{q-\alpha}{Q(h_n)}\Bigr)^{1/4}.
$$
Thus, $Q(\beta_g g_n)=\alpha$ and $Q(\beta_h h_n)=q-\alpha$.
Since $$|Q(g_n)-\alpha|<\varepsilon\quad \text{and}\quad |Q(h_n)-(q-\alpha)|<\varepsilon,$$ 
we have $$|\beta_g-1|\leq C\varepsilon\quad \text{and}\quad |\beta_h-1|\leq C\varepsilon,$$ 
where $C>0$ is independent of $n$. Then there exists a constant $C_1>0$, independent of $n$, such that
$$
I_\alpha 
 \leq E(\beta_g g_n) 
 \leq E(g_n)+C_1\varepsilon,
$$
and
$$I_{q-\alpha}
 \leq E(\beta_h h_n)
 \leq E(h_n)+C_1\varepsilon.$$
From these observations and Lemma \ref{lem:split.},
it follows that there exists a subsequence
$\{\psi_{n_k}\}$  and corresponding functions
$g_{n_k}$ and $h_{n_k}$ such that 
$$
E(g_{n_k}) \geq I_\alpha - \frac{1}{k}\quad \text{and}\quad E(h_{n_k}) \geq I_{q-\alpha} - \frac{1}{k},
$$
and 
$$E(\psi_{n_k})
  \geq E(g_{n_k}) + E(h_{n_k}) - \frac{1}{k}.$$
for all $k$. Hence,
$$
E(\psi_{n_k})
 \geq I_\alpha + I_{q-\alpha} - \frac{3}{k}.
$$
Taking the limit as $k\to\infty$ yields the desired inequality.
\end{proof}
Corollary \ref{cor: superaddit.} contradicts \eqref{H1L4 subadditivity}, and thus the case
$$
0<\alpha<q
$$
is impossible. Hence, the dichotomy alternative is ruled out. We finally exclude vanishing in the next result.

\begin{lemma}\label{lem:no-vanishing 1}
Let $\{\psi_n\}$ be a minimizing sequence for $I_q$. 
Then there exist a constant $\eta>0$ and a sequence $\{y_n\}$ in $\mathbb{R}$ such that
$$
\int_{y_n-1/2}^{ y_n+1/2} |\psi_n|^6dx \geq \eta
\qquad\text{for all } n.
$$
In particular, $\alpha>0$.
\end{lemma}

\begin{proof}

Recall that $\Lambda_1=\max\{\Lambda,(4q)^{1/4}\}$. Then
Lemma \ref{lem:bound-nonvanish.} provides
\begin{equation*}
|\psi_n|_{\dot{H}^1}+\|\psi_n\|_{L^4}\leq \Lambda_1,
\end{equation*} 
while Lemma \ref{lem:conc.-H1L4} yields  $\eta>0$ and $y_n\in\mathbb{R}$ such that
\begin{equation}\label{local-L6-lower}
\int_{y_n-1/2}^{ y_n+1/2} |\psi_n|^6dx \geq \eta 
\qquad\text{for every } n.
\end{equation}
Let $J_n:=(y_n-1/2, y_n+1/2)$.
Moreover, the  
Gagliardo-Nirenberg inequality  gives
\begin{equation}\label{GN-upper}
\int_{J_n} |\psi_n|^6dx
\;\le\;
C\Big(
\Lambda_1^{2/3}\|\psi_n\|_{L^4(J_n)}^{16/3}
\;+\;
\|\psi_n\|_{L^4(J_n)}^{6}
\Big).
\end{equation}
Combining \eqref{local-L6-lower} with \eqref{GN-upper}, we arrive at
$$
\eta
\leq
C\Big(
\Lambda_1^{2/3}\|\psi_n\|_{L^4(J_n)}^{16/3}
+\|\psi_n\|_{L^4(J_n)}^{6}
\Big).
$$
Since 
$$
\int_{y_n-1/2}^{ y_n+1/2}|\psi_n|^4dx
\leq M_n(1/2),
$$
it follows that
$$
\eta
\leq 
C\Big(
\Lambda_1^{2/3} M_n(1/2)^{4/3}
+ M_n(1/2)^{3/2}
\Big).
$$
Then $\eta>0$  forces
$$
M_n(1/2)\geq C_1>0
\qquad\text{for all sufficiently large }n,
$$
where $C_1$ depends only on $B$ and $\eta$.  
Hence,
$$
M(1/2)=\limsup_{n\to\infty} M_n(1/2)>0.
$$
Because $M(r)$ is nondecreasing in $r$, we obtain
$$
\alpha=\lim_{r\to\infty} M(r)\geq M(1/2)\geq c>0.
$$
\end{proof}

\section{Nonexistence of traveling wave solutions}\label{section 7}
\begin{theorem}\label{thm: nonexist.}
Let $\omega,c\in\mathbb{R}$. Then there is no nontrivial solution to
\eqref{nonlocal DNLS} of the form
\begin{equation}\label{eq: trav. wave}
u(x,t)
=\psi(x-ct)\exp\bigg(
i\omega t-i\frac{c}{2}(x-ct)
+\frac{i(\alpha+\beta)}{4}
\int_{x_0}^{x-ct}|\psi(y)|^2 dy
\bigg),
\end{equation}
where
$$
x_0=
\begin{cases}
0, & \text{in the critical case},\\
-\infty, & \text{in the subcritical case},
\end{cases}
$$
if one of the following conditions holds:
\begin{enumerate}
    \item $B\geq0$, $A\geq0$, and $\gamma>0$;
    \item $B\leq0$, $A\leq0$, $\gamma\in\mathbb{R}$, and
    $\omega+c^2/4\geq0$.
\end{enumerate}
Here $\psi$ is a weak solution belonging to
$$
\psi\in H^1(\mathbb{R};\mathbb{R})
$$
in the subcritical case, and
$$
\psi\in \dot H^1(\mathbb{R};\mathbb{R})\cap L^4(\mathbb{R};\mathbb{R})
$$
in the critical case.
\end{theorem}

\begin{remark}
In the critical case, the term
$$
\left(\omega+\frac{c^2}{4}\right)\|\psi\|_{L^2}^2
$$
is understood to be identically zero, without requiring
$\psi\in L^2(\mathbb{R})$.
\end{remark}

\begin{proof}
 We start with the first statement.
Suppose, by contradiction, that $u$ given by \eqref{eq: trav. wave}
is a nontrivial solution of \eqref{nonlocal DNLS}. Then its real-valued
profile $\psi$ satisfies 
\begin{equation}
-\psi'' - \left( \omega + \frac{c^2}{4} \right)\psi +  A |\psi|^2\psi + B |\psi|^4\psi + \gamma (\mathcal{H}(|\psi|^2))' \psi = 0.
\label{eq:main 1}
\end{equation}
Using the weak formulation of \eqref{eq:main 1}  and treating $\psi$ as a test function, we have
\begin{equation}\label{eq:energy1}
\begin{aligned}
\int_{\mathbb{R}} |\psi'|^2dx &= \left( \omega + \frac{c^2}{4} \right) \int_{\mathbb{R}} |\psi|^2dx -  A \int_{\mathbb{R}} |\psi|^4dx- B \int_{\mathbb{R}} |\psi|^6dx \\
& - \gamma \int_{\mathbb{R}} |\psi|^2(\mathcal{H}(|\psi|^2))'  dx.
\end{aligned}
\end{equation}
Testing \eqref{eq:main 1} against a suitable cutoff approximation of $x\psi'$ and passing to the limit, we obtain the Pohozaev identity
\begin{equation}\label{eq:energy2}
\begin{aligned}
\int_{\mathbb{R}} (\psi')^2 dx+\left(\omega+\frac{c^2}{4}\right) \int_{\mathbb{R}} |\psi|^2dx&-\frac{ A}{2} \int_{\mathbb{R}} |\psi|^4dx-\frac{B}{3} \int_{\mathbb{R}} |\psi|^6dx\\
&+\gamma \int_{\mathbb{R}} x\left(|\psi|^2\right)^{\prime} (\mathcal{H}(|\psi|^2))'dx=0.
\end{aligned}
\end{equation}
As shown in \cite[Appendix C]{gerard2024},
\begin{equation}\label{eq:integral-is-0}
\int_{\mathbb R}
x(|\psi|^2)'\mathcal H\big((|\psi|^2)'\big) dx=0.
\end{equation}
Thus, \eqref{eq:energy2} leads to
\begin{equation}\label{eq:pos.}
\left( \omega + \frac{c^2}{4} \right) \int_{\mathbb{R}} |\psi|^2  dx <  
 \frac{ A}{2} \int_{\mathbb{R}} |\psi|^4  dx 
+ \frac{B}{3} \int_{\mathbb{R}} |\psi|^6  dx.
\end{equation}
Combining \eqref{eq:energy1}--\eqref{eq:pos.}, we have
\begin{equation*}
\begin{aligned}
\left( \omega + \frac{c^2}{4} \right) \int_{\mathbb{R}} |\psi|^2  dx 
&- \frac{\gamma}{2} \int_{\mathbb{R}} |\psi|^2   \left(\mathcal{H}(|\psi|^2)\right)^{\prime}  dx  \\
&= \frac{ A}{2} \int_{\mathbb{R}} |\psi|^4  dx 
+ \frac{B}{3} \int_{\mathbb{R}} |\psi|^6  dx 
+ \frac{ A}{4} \int_{\mathbb{R}} |\psi|^4  dx 
+ \frac{B}{3} \int_{\mathbb{R}} |\psi|^6  dx \\
& > \left( \omega + \frac{c^2}{4} \right) \int_{\mathbb{R}} |\psi|^2  dx 
+ \frac{ A}{4} \int_{\mathbb{R}} |\psi|^4  dx 
+ \frac{B}{3} \int_{\mathbb{R}} |\psi|^6  dx.\\
\end{aligned}
\label{eq:final_inequality}
\end{equation*}
Therefore,
$$
0 > -\frac{\gamma}{2} \int_{\mathbb{R}} |\psi|^2   \left(\mathcal{H}(|\psi|^2)\right)^{\prime}  dx 
> \frac{ A}{4} \int_{\mathbb{R}} |\psi|^4  dx 
+ \frac{B}{3} \int_{\mathbb{R}} |\psi|^6  dx 
\geq  0,
$$
which is a contradiction.

Now we prove the second statement. Suppose, by contradiction, $u$ given by \eqref{eq: trav. wave} is a solution of \eqref{nonlocal DNLS}. Then $\psi$ solves \eqref{eq:main 1}. Again, testing \eqref{eq:main 1} against a suitable cutoff approximation of $x\psi'$ and passing to the limit, we obtain \eqref{eq:energy2}. Thus, we arrive at a contradiction
\begin{equation*}
0 < \int_{\mathbb{R}} |\psi'|^2  dx 
= -\left( \omega + \frac{c^2}{4} \right) \int_{\mathbb{R}} |\psi|^2  dx 
+ \frac{ A}{2} \int_{\mathbb{R}} |\psi|^4  dx 
+ \frac{B}{3} \int_{\mathbb{R}} |\psi|^6  dx \leq 0.
\end{equation*}
\end{proof}

\begin{theorem}\label{thm:critical-nonexistence-negative-gamma}
Let $\gamma<0$ and $B>\gamma^2/4$. Suppose that
$$
\psi\in \dot H^1(\mathbb{R};\mathbb{R})\cap L^6(\mathbb{R};\mathbb{R})
$$
is a weak solution of
\begin{equation}\label{eq:GL-stationary}
-\psi''
+
B|\psi|^4\psi
+
\gamma\psi |D|(|\psi|^2)
=
0.
\end{equation}
Then
$$
\psi\equiv0.
$$
\end{theorem}

\begin{proof}
Multiplying the equation by $\psi$ and integrating over $\mathbb{R}$, we obtain
$$
\int_{\mathbb{R}}|\psi'|^2 dx
+
B\int_{\mathbb{R}}|\psi|^6 dx
+
\gamma
\int_{\mathbb{R}}
|\psi|^2|D|(|\psi|^2) dx
=
0.
$$
Testing the equation against a suitable cutoff approximation of $x\psi'$ and passing to the limit, we obtain the Pohozaev identity
\begin{equation*}
B\int_{\mathbb{R}}|\psi|^6 dx
=
3\int_{\mathbb{R}}|\psi'|^2 dx,
\end{equation*}
where we have used \eqref{eq:integral-is-0}.
Substituting this identity into the first one gives
\begin{equation}\label{eq:after-Pohozaev}
\int_{\mathbb{R}}
|\psi|^2|D|(|\psi|^2) dx
=
\frac{4}{-\gamma}
\int_{\mathbb{R}}|\psi'|^2 dx.
\end{equation}
Using Proposition \ref{prop:sharp-nonlocal-inequality} and
\eqref{eq:after-Pohozaev},
we obtain
\begin{equation}\label{eq:1}
\frac{16}{\gamma^2}
\left(
\int_{\mathbb{R}}|\psi'|^2 dx
\right)^2
\le
\frac{4}{B}
\left(
\int_{\mathbb{R}}|\psi'|^2 dx
\right)^2.
\end{equation}
Since $\psi$ is nontrivial, it follows that
$$
\int_{\mathbb{R}}|\psi'|^2 dx>0.
$$
Indeed, $\psi'=0$ a.e. would imply that $\psi$ is constant, and
$\psi\in L^6(\mathbb{R};\mathbb{R})$ would then force $\psi\equiv0$. Therefore, 
$$
B>\frac{\gamma^2}{4}
$$
together with \eqref{eq:1}
force
$$
\psi\equiv0.
$$
\end{proof}

\section{Appendix}

\begin{proposition}\label{prop:sharp-nonlocal-inequality}
For every
$$
\psi\in \dot{H}^1(\mathbb{R} ; \mathbb{C}) \cap L^6(\mathbb{R} ; \mathbb{C}) ,
$$
we have
$$
\big\||D|^{1/2}(|\psi|^2)\big\|_{L^2}^2
\le
\frac{2}{\sqrt3}\,
\|\psi'\|_{L^2}\,
\|\psi\|_{L^6}^3.
$$
Moreover, equality is attained for 
$$
\psi(x)
=
\frac{\sqrt2}{\sqrt{1+x^2}}.
$$
In particular, the constant $2/\sqrt3$ is sharp.
\end{proposition}

\begin{proof}
 We first assume that $\psi\in C_c^\infty(\mathbb{R}; \mathbb{C})$ and set
$$
f:=|\psi|^2.
$$
Since $|D|f=(\mathcal{H}f)'$, integration by parts gives
$$
\begin{aligned}
\big\||D|^{1/2}f\big\|_{L^2}^2
&=
-\int_{\mathbb{R}}f'\mathcal{H}f dx\\
&=
-2\operatorname{Re}
\int_{\mathbb{R}}
\psi'\overline{\psi}\,
\mathcal{H}(|\psi|^2) dx.
\end{aligned}
$$
Hence, by the Cauchy--Schwarz inequality,
\begin{equation}\label{eq:Cauchy-Schwarz}
\big\||D|^{1/2}(|\psi|^2)\big\|_{L^2}^2
\le
2\|\psi'\|_{L^2}
\left(
\int_{\mathbb{R}}
|\psi|^2
\big(\mathcal{H}(|\psi|^2)\big)^2 dx
\right)^{1/2}.
\end{equation}
We now use the Tricomi identity and the skew-adjointness of the
Hilbert transform; see, for example,
\cite[Sections 4.11 and 4.16]{King2009}. Taking $g=f$ in the
Tricomi identity 
$$
H\{f(x) H g(x)+g(x) H f(x)\}=H f(x) H g(x)-f(x) g(x) .
$$
gives
$$
2\mathcal{H}(f\mathcal{H}f)
=
(\mathcal{H}f)^2-f^2.
$$
Multiplying by $f$, integrating over $\mathbb{R}$, and using skew-adjointness, we obtain
$$
\begin{aligned}
\int_{\mathbb{R}}f(\mathcal{H}f)^2 dx
-
\int_{\mathbb{R}}f^3 dx
&=
2\int_{\mathbb{R}}f\,\mathcal{H}(f\mathcal{H}f) dx\\
&=
-2\int_{\mathbb{R}}f(\mathcal{H}f)^2 dx.
\end{aligned}
$$
Therefore,
$$
\int_{\mathbb{R}}f(\mathcal{H}f)^2 dx
=
\frac{1}{3}\int_{\mathbb{R}}f^3 dx,
$$
cf \cite[Appendix C]{gerard2024}. Combining the last identity with \eqref{eq:Cauchy-Schwarz}, we arrive at
$$
\big\||D|^{1/2}(|\psi|^2)\big\|_{L^2}^2
\le
\frac{2}{\sqrt3}
\|\psi'\|_{L^2}
\|\psi\|_{L^6}^3.
$$

By density argument,
there exists a sequence $\psi_n\in C_c^\infty(\mathbb{R}; \mathbb{C})$ such that
$$
\psi_n'\to\psi'
\quad\text{in }L^2(\mathbb{R}; \mathbb{C}),
\qquad
\psi_n\to\psi
\quad\text{in }L^6(\mathbb{R}; \mathbb{C}).
$$
In particular,
$$
|\psi_n|^2\to|\psi|^2
\quad\text{in }L^3(\mathbb{R}; \mathbb{C}),
$$
because
$$
\big\||\psi_n|^2-|\psi|^2\big\|_{L^3}
\le
\|\psi_n-\psi\|_{L^6}
\big(
\|\psi_n\|_{L^6}+\|\psi\|_{L^6}
\big).
$$
Moreover,
$$
\big(|D|^{1/2}(|\psi_n|^2)\big)
$$
is bounded in $L^2(\mathbb{R}; \mathbb{C})$. Hence, up to a subsequence,
$$
|D|^{1/2}(|\psi_n|^2)
\rightharpoonup g
\quad\text{weakly in }L^2(\mathbb{R}; \mathbb{C}).
$$
Now we show that 
$$
g=|D|^{1/2}(|\psi|^2)
$$
in the sense of distributions. For every $\varphi\in\mathcal{S}(\mathbb{R}; \mathbb{C})$,
$$
\begin{aligned}
\langle g,\varphi\rangle
&=
\lim_{n\to\infty}
\big\langle
|D|^{1/2}(|\psi_n|^2),\varphi
\big\rangle\\
&=
\lim_{n\to\infty}
\big\langle
|\psi_n|^2,|D|^{1/2}\varphi
\big\rangle\\
&=
\big\langle
|\psi|^2,|D|^{1/2}\varphi
\big\rangle.
\end{aligned}
$$
The last limit follows from
$|\psi_n|^2\to|\psi|^2$ in $L^3(\mathbb{R};\mathbb{C})$ and
$|D|^{1/2}\varphi\in L^{3/2}(\mathbb{R};\mathbb{C})$. Indeed,
since $|D|^{1/2}=(-\Delta)^{1/4}$,
\cite[Appendix~B]{Abatangelo2019} yields
$$
|D|^{1/2}\varphi(x)
=
O(|x|^{-3/2})
\qquad\text{as }|x|\to\infty,
$$
while $|D|^{1/2}\varphi$ is locally bounded.
Therefore,
$$
g=|D|^{1/2}(|\psi|^2)
$$
in $\mathcal{S}'(\mathbb{R}; \mathbb{C})$. Then, by weak lower semicontinuity,
$$
\begin{aligned}
\big\||D|^{1/2}(|\psi|^2)\big\|_{L^2}^2
&\le
\liminf_{n\to\infty}
\big\||D|^{1/2}(|\psi_n|^2)\big\|_{L^2}^2\\
&\le
\frac{2}{\sqrt3}
\|\psi'\|_{L^2}
\|\psi\|_{L^6}^3.
\end{aligned}
$$

Finally, the constant $2/\sqrt3$ is sharp, since
$$
\psi(x)
=
\frac{\sqrt2}{\sqrt{1+x^2}}
$$
belongs to $\dot H^1(\mathbb{R}; \mathbb{C})\cap L^6(\mathbb{R}; \mathbb{C})$ and attains equality.
\end{proof}

\section{Acknowledgments}
This research was funded by the Science Committee of the Ministry of Science and Higher Education of Kazakhstan (Grant No. AP26194665) and the Nazarbayev University Faculty Development Competitive Research Grants Program 040225FD4702. 

\section{Conflict of interest}
The authors declare that they have no conflict of interest.

\section{Data availability}
The manuscript has no associated data.

\addcontentsline{toc}{chapter}{Bibliography}
\bibliography{refs}      

\bibliographystyle{abbrv}  

\end{document}